\documentclass[a4paper,11pt,reqno]{amsart}
\usepackage[foot]{amsaddr}

\usepackage{amssymb}
\usepackage{epsfig}
\usepackage{amsfonts}
\usepackage{amsmath}
\usepackage{euscript}
\usepackage{amscd}
\usepackage{amsthm}
\usepackage{enumitem}
\DeclareMathAlphabet{\mathpzc}{OT1}{pzc}{m}{it}
\usepackage{enumitem}
\usepackage{color}
\usepackage{mathtools}
\usepackage[pagebackref, hypertexnames=false, colorlinks, citecolor=red, linkcolor=blue, urlcolor=red]{hyperref}

\usepackage{marginnote}

\usepackage[normalem]{ulem}
\newcommand{\marginextend}[1]{ \addtolength{\oddsidemargin}{-#1}  \addtolength{\evensidemargin}{-#1}
  \addtolength{\textwidth}{#1}\addtolength{\textwidth}{#1}}
\newcommand{\updownextend}[1]{ \addtolength{\topmargin}{-#1}  \addtolength{\textheight}{#1}
\addtolength{\textheight}{#1}}
\marginextend{1cm}
\updownextend{0cm}
\allowdisplaybreaks[4]

\usepackage{mathrsfs}
\DeclareFontFamily{OT1}{pzc}{}
\DeclareFontShape{OT1}{pzc}{m}{it}{<-> s * [1.10] pzcmi7t}{}
\DeclareMathAlphabet{\mathpzc}{OT1}{pzc}{m}{it}

\DeclareSymbolFont{SY}{U}{psy}{m}{n}
\DeclareMathSymbol{\emptyset}{\mathord}{SY}{'306}

\theoremstyle{plain}

\newtheorem{thm}{Theorem}[section]
\newtheorem*{thm*}{Theorem}
\newtheorem{cor}[thm]{Corollary}
\newtheorem{lem}[thm]{Lemma}
\newtheorem{prop}[thm]{Proposition}
\newtheorem{defn}[thm]{Definition}
\newtheorem{rem}[thm]{Remark}

\newtheoremstyle{named}{}{}{\itshape}{}{\bfseries}{.}{.5em}{#1 \thmnote{#3}}
\theoremstyle{named}

\numberwithin{equation}{section}

\def\C{{\mathbb C}}

\def\norm#1{\left\|{#1}\right\|}

\def\l{\lambda}

\def\ov{\overline}

\def\m{\mathcal}
\def\mb{\mathbb}

\def\beq{\begin{eqnarray}}
\def\eeq{\end{eqnarray}}
\def\beqa{\begin{eqnarray*}}
\def\eeqa{\end{eqnarray*}}

\def\ov{\overline}
\def\bl{\boldsymbol}

\def\sgn{{\rm sgn}}
\def\bl{\boldsymbol}
\newcommand{\overbar}[1]{\mkern 1.5mu\overline{\mkern-1.5mu#1\mkern-1.5mu}\mkern 1.5mu}

\newcommand{\be}{\begin{equation}}
\newcommand{\ee}{\end{equation}}
\newcommand{\bea}{\begin{eqnarray}}
\newcommand{\eea}{\end{eqnarray}}
\newcommand{\Bea}{\begin{eqnarray*}}
\newcommand{\Eea}{\end{eqnarray*}}
\newcommand{\inner}[2]{\langle #1,#2 \rangle }%

\newcounter{cnt1}
\newcounter{cnt2}
\newcounter{cnt3}
\newcommand{\blr}{\begin{list}{$($\roman{cnt1}$)$}
 {\usecounter{cnt1} \setlength{\topsep}{0pt}
 \setlength{\itemsep}{0pt}}}
\newcommand{\bla}{\begin{list}{$($\alph{cnt2}$)$}
 {\usecounter{cnt2} \setlength{\topsep}{0pt}
 \setlength{\itemsep}{0pt}}}
\newcommand{\bln}{\begin{list}{$($\arabic{cnt3}$)$}
 {\usecounter{cnt3} \setlength{\topsep}{0pt}
 \setlength{\itemsep}{0pt}}}
\newcommand{\el}{\end{list}}

\DeclareMathOperator  {\adj}{adj}

\title{The weighted Bergman spaces and complex reflection groups}
\author[Ghosh]{Gargi Ghosh}
\email[Ghosh]{gargi.ghosh@uj.edu.pl}
\address[Ghosh]{Jagiellonian University, Faculty of Mathematics and Information Technologies, 30-348 Krakow, Poland}
\begin{document} \maketitle

\begin{abstract}
    We consider a bounded domain $\Omega \subseteq \mathbb C^d$ which is a $G$-space for a finite complex reflection group $G$. For each one-dimensional representation of the group $G,$ the relative invariant subspace of the weighted Bergman space on $\Omega$ is isometrically isomorphic to a weighted Bergman space on the quotient domain $\Omega/G.$ Consequently, formulae involving the weighted Bergman kernels and projections of $\Omega$ and $\Omega /G$ are established. As a result, a transformation rule for the weighted Bergman kernels under a proper holomorphic mapping with $G$ as its group of deck transformations is obtained in terms of the character of the sign representation of $G$. Explicit expressions for the weighted Bergman kernels of several quotient domains (of the form $\Omega / G$) have been deduced to demonstrate the merit of the described formulae.
\end{abstract}
\section{Introduction}
Suppose that $G$ is a finite group acting on a bounded domain $\Omega\subseteq \mb C^d.$ Then $\Omega$ is said to be a $G$-invariant domain or a $G$-space. 
In this article, we focus on the following question:
\begin{itemize}
\item For any $G$-space $\Omega,$ how the weighted Bergman kernels and weighted Bergman projections of $\Omega$ and $\Omega / G$ are related? 
\end{itemize}
To achieve our goal, we first observe that it is not obvious that $\Omega/G$ is a domain for a $G$-invariant domain $\Omega$. Although it is known that $\Omega/G$ can be given the structure of a complex analytic space which is biholomorphically equivalent to some domain in $\mb C^d$ whenever $G$ is a finite complex reflection group \cite{MR807258} \cite[Subsection 3.1.1]{BDGS} \cite[Proposition 1]{MR2542964}. Henceforth, we confine our attention to a finite complex reflection group $G$ to establish formulae between the weighted Bergman kernels of the domain $\Omega$ and the quotient domain $\Omega / G$ via the one-dimensional representations of $G.$ This in turn invokes identities involving associated weighted Bergman projections of $\Omega$ and $\Omega / G.$ 
For the sign representation of $G,$ the transformation formula for the weighted Bergman kernels and the weighted Bergman projections generalize Bell's transformation rule for the Bergman kernels under proper holomorphic maps in \cite[p. 687, Theorem 1]{MR645338} to the weighted Bergman kernels and the Bergman projection formula in \cite[p. 167, Theorem 1]{MR610182} to weighted Bergman projections.  As a result, we are able to overcome the limitation of removing the critical points of the proper holomorphic map in the transformation rule for the Bergman kernel. 
However, our approach is entirely different from that of Bell. We largely use various tools from invariant theory of finite complex reflection groups and an analytic version of  well known Chevalley-Shephard-Todd theorem obtained in \cite{BDGS} by the author and her collaborators. We start by recalling notions involving finite complex reflection groups which are necessary to state the main results.


A \emph{complex reflection} on $\C^d$ is a linear homomorphism $\sigma: \C^d \rightarrow \C^d$ such that $\sigma$ has finite order in $GL(d,\mb C)$ and the rank of $({\rm id} - \sigma)$ is 1. A group generated by complex reflections is called a complex reflection group. For example, any finite cyclic group, the symmetric group $\mathfrak{S}_d$ on $d$ symbols, the dihedral groups are complex reflection groups \cite{MR2542964}. A complex reflection group $G$ acts on $\mb C^d$ by (right action) \bea\label{action}\sigma \cdot \bl z = \sigma^{-1}\bl z, \text{ for } \sigma \in G \text{ and }\bl z \in \mb C^d\eea and the group action extends to the set of all complex-valued functions on $\mb C^d$ by $\sigma( f)(\bl z) =  f({\sigma}^{-1}\cdot \bl z), \,\, \text{~for~} \sigma \in G \text{~and~} \bl z \in \mb C^d.$  We call $f$ a $G$-invariant function if $\sigma(f) =f$ for all $\sigma \in G.$ There is a system of $G$-invariant  algebraically independent homogeneous polynomials $\{\theta_i\}_{i=1}^d$ associated to a complex reflection group $G,$ called a homogeneous system of  parameters (hsop) or basic polynomials associated to $G.$ A finite complex reflection group $G$ is characterized by the fact that the ring of $G$-invariant polynomials is a polynomial ring generated by some hsop $\{\theta_i\}_{i=1}^d$ associated to $G$ \cite[p.282]{Shephard-Todd} (cf. Theorem \ref{cst}). The map $\bl \theta:=(\theta_1,\ldots,\theta_d): \mb C^d \to \mb C^d$ is said to be a basic polynomial map associated to the finite complex reflection group $G$ \cite{BDGS,MR3133729}. 

If $\Omega$ is a $G$-space under the action defined in Equation \eqref{action} then $\bl \theta(\Omega)$ is a domain and the quotient topological space $\Omega/G$ is a complex analytic space biholomorphic to $\bl \theta(\Omega).$ Therefore, without loss of generality we work with the domain $\bl \theta(\Omega)$ instead of $\Omega / G.$  Special cases of such quotient domains have been studied in many instances. 
For example, $\mb D^d / \mathfrak{S}_d$ ($\mb D^d$ = the cartesian product of $d$ copies of the unit disc $\mb D$ in the complex plane) can be realized as the symmetrized polydisc (denoted by $\mb G_d$), the tetrablock is biholomorphic to the quotient domain $\m R_{II}/\mathfrak{S}_2$ where $\m R_{II}$ is the classical Cartan domain of second type \cite{MR3133729} and following \cite{bender2020lpregularity}, a monomial polyhedron can be realized as a quotient domain $\Omega / G$ for $\Omega \subseteq \mb D^d$ and a finite abelian group $G.$ 
The symmetrized polydisc has been studied extensively in the last two decades in the context of function theory \cite{MR3771126,MR2135687,MR3043017}, operator theory \cite{MR1674635,MR1744711,MR3188714, MR2142182} and geometry \cite{MR3912883,MR4293930,MR2077158}. Also, the study of geometry and function theory on tetrablock is currently an active area of interest \cite{MR2736338, MR3107680,MR3511461,MR2365665,MR2418303}. Recently, in \cite{MR4088498} the symmetrized polydisc and in \cite{bender2020lpregularity} the monomial polyhedron were considered in connection with $L^p$-regularity of the Bergman projection and it turns out that Theorem \ref{maint} is a key tool to study $L^p$-regularity of the weighted Bergman projections on quotient domains \cite{gho-gho}. 

For an analytic Hilbert module $\m H$ on $\Omega,$ the relative invariant subspace associated to a one-dimensional representation $\varrho$ of $G,$ is defined by
\Bea
R^G_{\varrho}(\m H) = \{f \in \m H : \sigma(f)  = \chi_\varrho(\sigma) f ~ {\rm for ~~ all~} \sigma \in G \},
\Eea
 where $\chi_\varrho$ denotes the character of the representation $\varrho$. We show that the subspace $R^G_{\varrho}(\m H)$ is a reproducing kernel Hilbert space (cf. Lemma \ref{ortho}). 
The elements  of $R_\varrho^G(\m H)$ are said to be $\varrho$-invariant.  There exists a polynomial (unique up to a constant multiple), say $\ell_\varrho,$ which forms a basis of the ring of $\varrho$-invariant polynomials as a free module over the ring of $G$-invariant polynomials in $d$ variables \cite[p. 139, Theorem 3.1]{MR460484}. A characterization of the relative invariant subspace is that every $f \in R_\varrho^G(\m H)$ is divisible by the polynomial $\ell_\varrho$ and the quotient is in the ring of $G$-invariant holomorphic functions on $\Omega$ (cf. Lemma \ref{quo}). The polynomial $\ell_\varrho$ plays an important role in our discussion. An explicit expression for $\ell_\varrho$ (unique up to a constant multiple) has been obtained from the representation $\varrho$ \cite[p. 139, Theorem 3.1]{MR460484} (cf. Lemma \ref{gencz}). 

We now briefly describe the results we have proved in this article.
\subsection{On the weighted Bergman kernels}
    Given a continuous weight function $\omega: \Omega \to (0,\infty),$ $L_\omega^2(\Omega)$ denotes the Hilbert space of Lebesgue measurable functions (equivalence classes of functions) on $\Omega$ which are square integrable with respect to the measure $\omega(\bl z)dV(\bl z),$ where $dV$ is the normalized Lebesgue measure on $\Omega.$ The weighted Bergman space $\mb A^2_\omega(\Omega)$ is the closed subspace consisting of holomorphic functions in $L_\omega^2(\Omega)$. For $\omega \equiv 1,$ $\mb A^2_\omega(\Omega)$ reduces to the Bergman space $\mb A^2(\Omega).$ 
    We consider a weight function of the form $\omega =\widetilde{\omega}\circ \bl \theta$ for a continuous map $\widetilde{\omega} : \bl \theta(\Omega) \to (0,\infty).$  For each one-dimensional representation $\varrho$ of $G,$ we set $\omega_\varrho(\bl \theta(\bl z)) = \frac{|\ell_\varrho(\bl z)|^2}{|J_{\bl \theta}(\bl z)|^2} \widetilde{\omega}(\bl \theta(\bl z)),$ where $J_{\bl \theta}$ is the determinant of the complex jacobian matrix of the basic polynomial map $\bl \theta.$ The following theorem is the main result of this article.
\begin{thm}\label{firstmain}
Let $G$ be a finite complex reflection group and a bounded domain $\Omega \subseteq \mb C^d$ be a $G$-space. Then for each one-dimensional representation $\varrho$ of $G,$ the relative invariant subspace  $R^G_{\varrho}(\mb A^2_\omega(\Omega))$ is isometrically isomorphic to the weighted Bergman space $\mb A^2_{\omega_\varrho}(\bl \theta(\Omega)),$ where $\bl \theta$ is a basic polynomial map associated to the group $G.$

Moreover, the reproducing  kernel $\m B_{\omega_\varrho}$ of the weighted Bergman space $\mb A^2_{\omega_\varrho}\big(\bl \theta(\Omega)\big)$ is given by
\bea\label{formu} \m B_{\omega_\varrho}\big(\bl \theta(\bl z), \bl \theta(\bl w)\big) = \frac{1}{\ell_\varrho(\bl z) \ov{\ell_\varrho(\bl w)}} \displaystyle\sum_{\sigma \in G} \chi_\varrho(\sigma^{-1}) \m B_\omega(\sigma^{-1} \cdot \bl z, \bl w) \text{~~for ~~}\bl z, \bl w \in \Omega,\eea
where $\m B_\omega$ is the reproducing kernel of the weighted Bergman space $\mb A^2_\omega(\Omega)$, $\chi_\varrho$ denotes the character of the representation $\varrho$ and $\ell_\varrho$ is as described in Lemma \ref{gencz}. 
\end{thm}
We provide a number applications of Theorem \ref{firstmain} in Section \ref{app} to determine explicit formulae for the weighted Bergman kernels of various quotient domains. For example, the weighted Bergman space $\mb A^2_\omega(\mb D^d)$ with the $\mathfrak{S}_d$-invariant weight function $\omega(\bl z) = \prod_{j=1}^d (1-|z_j|^2)^{\l -2}, \l>1,$ has two relative invariant subspaces 
$\mb A^2_{\omega,\rm anti}(\mb D^d) = \{f \in \mb A^2_\omega(\mb D^d): \sigma(f) = {\rm sign}(\sigma) f \text{~for~} \sigma \in \mathfrak{S}_d\}$ and $\mb A^2_{\omega,\rm sym}(\mb D^d) = \{f \in \mb A^2_\omega(\mb D^d): \sigma(f) = f \text{~for~} \sigma \in \mathfrak{S}_d\}$ associated to the only one-dimensional representations (which are trivial representation and sign representation) of $\mathfrak{S}_d.$ The subspace $\mb A^2_{\omega,\rm anti}(\mb D^d)$ is isometrically isomorphic to some weighted Bergman space on $\mb G_d$ with the reproducing kernel $$\m B_{\widetilde{\omega}}\big(\bl s(\bl z), \bl s(\bl w)\big)=\frac{\det \big(\!\!\big( (1 - z_i \bar{w_j})^{-\l}\big)\!\!\big)_{i,j =1 }^{d}}{\displaystyle\prod_{i<j} (z_i-z_j) (\overbar{w_i}-\overbar{w_j})},\,\, \bl z, \bl w \in \mb D^d,$$ and $\mb A^2_{\omega,\rm sym}(\mb D^d)$ is isometrically isomorphic to another weighted Bergman space on $\mb G_d$ whose reproducing kernel is given by $$\m B_{\widetilde{\omega}_1}(\bl s(\bl z), \bl s(\bl w)) = {\rm perm}\big(\!\!\big( (1 - z_i \bar{w_j})^{-\l}\big)\!\!\big)_{i,j =1 }^{d},\,\, \bl z, \bl w \in \mb D^d,$$ where ${\rm perm}A$ denotes the permanent of the matrix $A$ (cf. Proposition \ref{anti} and Proposition \ref{sym}).

\begin{enumerate}[leftmargin=*]
\item[1.]{\fontfamily{qpl}\selectfont The case of abelian groups.} If $G$ is a finite abelian group, each irreducible representation of $G$ is one-dimensional. In this case, the weighted Bergman kernel $\mb A^2_\omega(\Omega)$ is isometrically isomorphic to an orthogonal direct sum of weighted Bergman spaces on $\bl \theta(\Omega),$ that is, $$\mb A^2_\omega(\Omega) \cong \displaystyle\oplus_{\varrho \in \widehat{G}} \mb A^2_{\omega_\varrho}(\bl \theta(\Omega)),$$ where $\widehat{G}$ is the set of equivalence classes of irreducible representations of $G.$ The reproducing kernel of each weighted Bergman space $\mb A^2_{\omega_\varrho}(\bl \theta(\Omega))$ can be obtained from Equation \eqref{formu}. Therefore, we have the following identity involving weighted Bergman kernels:
\Bea
\m B_\omega(\bl z,\bl w) = \frac{1}{|G|} \sum_{\varrho \in \widehat{G}} \ell_\varrho(\bl z) \m B_{\omega_\varrho}\big(\bl \theta(\bl z),\bl \theta(\bl w)\big)\overline{\ell_\varrho(\bl w)}.
\Eea
Recently, in \cite{MR4244876} Nagel and Pramanik made an analogous observation when a basic polynomial map of some finite abelian group is given by a monomial type mapping. 
\item[2.]{\fontfamily{qpl}\selectfont The case of sign representation.} The one-dimensional representation,  $\sgn : G \to \mb C^*$ is defined by \bea\label{sign} \sgn(\sigma) = (\det(\sigma))^{-1},\eea see \cite{MR460484}. In the proceedings, we refer this representation as \emph{sign representation}. 
For the sign representation, Equation \eqref{formu} reduces to the following identity which is worth mentioning:
\bea\label{la}\m B_{\widetilde{\omega}}\big(\bl \theta(\bl z), \bl \theta(\bl w)\big) = \frac{1}{J_{\bl \theta}(\bl z) \ov{J_{\bl \theta}(\bl w)}} \displaystyle\sum_{\sigma \in G} \det(\sigma) \m B_\omega(\sigma^{-1} \cdot \bl z, \bl w)  \,\,\,\,\,\, \text{for}\,\,  \bl z, \bl w \in \Omega.
\eea
For $\omega \equiv 1,$ this identity emerges as a very convenient tool to determine explicit formula for the Bergman kernel of $\Omega / G$ in terms of the Bergman kernel of $\Omega.$ The Bergman kernel plays a crucial role in complex analysis and complex geometry. For instance, an explicit expression for the Bergman kernel of a domain is essential for understanding the boundary behaviour of geodesics arising from the Bergman metric of the domain and for characterizing Lu Qi-keng domains, see \cite{MR1775958,MR1469401,MR338454,MR0473215}.  We demonstrate the efficiency of Equation \eqref{la} by obtaining explicit formulae for the Bergman kernels for several domains, namely, symmetrized polydisc, monomial polyhedron, a subclass of complex ellipsoids and $\mb D^2 / D_{2k}$ for the dihedral group $D_{2k}$.
\item[3.]{\fontfamily{qpl}\selectfont Weighted Bergman kernel and proper holomorphic maps.} 
We consider a proper holomorphic map $\bl f :  \Omega \to \widetilde{\Omega}$ with $G$ as the group of deck transformations. Such a map $\bl f  $ is also called a proper holomorphic map \emph{factored by (automorphisms)} $G$ in \cite{MR807258,MR1131852}, see also \cite[p. 7]{BDGS}. In Proposition \ref{representative}, we show that  $\bl f : \Omega\to \widetilde{\Omega}$ can be written as $\bl f = \bl h \circ \bl \theta$ for a biholomorphism $\bl h : \bl \theta(\Omega) \to \widetilde{\Omega}$ and a basic polynomial map $\bl \theta : \Omega \to \bl \theta(\Omega)$ associated to the group $G.$ This observation leads us to the following transformation rule for the weighted Bergman kernels under a proper holomorphic map with $G$ as the group of deck transformations: \bea\label{lala}\m B_{\widetilde{\omega}}\big(\bl f(\bl z), \bl f(\bl w)\big) = \frac{1}{J_{\bl f}(\bl z) \ov{J_{\bl f}(\bl w)}} \displaystyle\sum_{\sigma \in G} \det(\sigma) ~\m B_\omega(\sigma^{-1} \cdot \bl z, \bl w)  \,\,\,\,\,\, \text{for}\,\,  \bl z, \bl w \in \Omega,\eea
where the weight is of the form $\omega= \widetilde{\omega}\circ \bl f$, $\m B_\omega$ and $\m B_{\widetilde{\omega}}$ are the reproducing kernels of $\mb A^2_\omega(\Omega)$ and $\mb A^2_{\widetilde{\omega}}(\widetilde{\Omega}),$ respectively. 
As expected, for $\omega \equiv 1,$ Equation \eqref{lala} overlaps with the transformation rule described by Steven Bell in \cite[p. 687, Theorem 1]{MR645338}. We emphasize that Equation \eqref{lala} works for the critical points of $\bl f$ as well.
\item[4.]{\fontfamily{qpl}\selectfont Bergman kernels for Rudin's domains.} A family of quotient domains of the form $\mb B_d / G$ is described in \cite{MR667790}, where $\mb B_d$ denotes the open unit ball with respect to the $\ell^2$-norm on $\mb C^d$ and the group $G$ is a conjugate to a finite complex reflection group. Following \cite[p. 427]{MR742433}, we refer such domains as Rudin's domains. The domain $\Omega \subset \mb C^d$ is a Rudin's domain if and only if there exists a proper holomorphic map $\bl F: \mb B_d \to \Omega.$ 
Then the Bergman kernel $\m B_\Omega$ of a Rudin's domain $\Omega$ is given by the following formula: 
\Bea
\m B_\Omega\big(\bl F(\bl z), \bl F(\bl w)\big)= \frac{1}{J_{\bl F}(\bl z) \ov{J_{\bl F}(\bl w)}} \displaystyle\sum_{\sigma \in G}  \frac{J_{\sigma^{-1}}(\bl z)}{\big(1 - \langle \sigma^{-1} \cdot\bl z, \bl w \rangle \big)^{d+1}},\,\, \bl z, \bl w \in \mb B_d,
\Eea
where 
$\langle \cdot, \cdot\rangle$ denotes the standard inner product in $\mb C^d.$
\end{enumerate}
 
\subsection{On the weighted Bergman projections} The weighted Bergman projection $P_\Omega^\omega: L^2_\omega(\Omega) \to \mb A^2_\omega(\Omega)$ is defined by $$(P_\Omega^\omega \phi)(\bl z) =\inner{\phi}{\m B_\omega(\cdot,\bl z)}  = \int_\Omega \phi(\bl w) \m B_\omega(\bl z,\bl w) \omega(\bl w) dV(\bl w), \,\, \phi \in L^2_\omega(\Omega),$$ where $\m B_\omega$ is the reproducing kernel of the weighted Bergman space $\mb A^2_\omega(\Omega).$ For a holomorphic function $\phi$ in $L^2_\omega(\Omega),$ we have $(P_\Omega^\omega \phi)(\bl z) =\inner{\phi}{\m B_\omega(\cdot,\bl z)} = \phi(\bl z).$ 
 \begin{thm}\label{maint}
 For each one-dimensional representation $\varrho$ of $G,$ the weighted Bergman projections $P_{\bl \theta(\Omega)}^{\omega_\varrho}$ and $P_\Omega^\omega$ are related by the following formula: 
 \bea\label{bpro}
P_\Omega^\omega\big(\ell_\varrho ~ (\phi \circ \bl \theta) \big)= \ell_\varrho\big((P_{\bl \theta(\Omega)}^{\omega_\varrho} \phi) \circ \bl \theta\big) , \,\, \phi \in   L^2_{\omega_\varrho}\big(\bl \theta(\Omega)\big),
\eea
where $P_{\bl \theta(\Omega)}^{\omega_\varrho} :  L^2_{\omega_\varrho}(\bl \theta(\Omega))\to \mb A^2_{\omega_\varrho}(\bl \theta(\Omega))$ is the weighted Bergman projection.
 \end{thm}
That is, we get $k$ formulae involving weighted Bergman projections, where $k=$ the number of one-dimensional representations in $\widehat{G}.$ For a fixed one-dimensional representation $\varrho \in \widehat{G},$ Equation \eqref{bpro} is true for the weight tuple $(\omega,\omega_\varrho)$ whenever $\omega_\varrho(\bl \theta(\bl z)) = \frac{|\ell_\varrho(\bl z)|^2}{|J_{\bl \theta}(\bl z)|^2} \omega(\bl z)$ for $\bl z \in \Omega$. For the sign representation of $G,$ Equation \eqref{bpro} generalizes Bell's formula for Bergman projections \cite[Equation 2.2, p. 686]{MR645338} to the weighted Bergman projections for the weight tuples $(\omega,\widetilde{\omega})$, where $\omega = \widetilde{\omega} \circ \bl \theta$. We expect to use Theorem \ref{bpro} in the study of $L^p$ regularity of the Bergman projections. Lastly, we mention a very recent paper in this direction \cite{MR4523519}.

\section{Complex reflection groups and analytic Hilbert modules}\label{sect2} We begin by recalling a number of useful definitions and standard results about complex reflection groups. 
\subsection{Chevalley-Shephard-Todd Theorem}
Let $G$ be a finite complex reflection group. Recall that $G$ acts on the set of functions on $\C^d$ by $\sigma (f)(\bl z)=f({\sigma}^{-1}\cdot \bl z).$ A function is said to be $G$-invariant if $\sigma(f)=f,$ for all $\sigma \in G.$ The ring of all complex polynomials in $d$ variables is denoted by $\mb C[z_1,\ldots,z_d]$. Moreover, the set of all $G$-invariant polynomials, denoted by $\mb C[z_1,\ldots,z_d]^G$, forms a subring and coincides with the relative invariant subspace $R^G_{\rm trivial}(\mb C[z_1,\ldots,z_d])$ associated to the trivial representation of $G$. Chevalley, Shephard and Todd characterize finite complex reflection groups in the following theorem.
\begin{thm}[CST Theorem]\label{cst}\cite[ p. 112, Theorem 3]{MR1890629}\label{A}
The invariant ring $\C[z_1,\ldots,z_d]^G$ is equal to $\C[\theta_1,\ldots,\theta_d]$, where $\theta_i$'s are algebraically independent
homogeneous polynomials if and only if $G$ is a finite complex reflection group.
\end{thm}
The collection of homogeneous polynomials $\{\theta_i\}_{i=1}^d$ is called a homogeneous system of parameters (hsop) or basic polynomials associated to the complex reflection group $G$. Although a hsop is not unique but the degrees of $\theta_i$'s are unique for $G$ up to order.
The map ${\bl\theta}: \C^d \rightarrow \C^d$, defined by
\bea\label{theta}
{\bl\theta}(\bl z) = \big(\theta_1(\bl z),\ldots,\theta_d(\bl z)\big),\,\,\bl z\in\C^d
\eea is called a basic polynomial map associated to the group $G.$ The following proposition shows that a complex reflection group $G$ always induces a canonical polynomial proper mapping.
\begin{prop}\label{domain}\cite[p. 16, Proposition 5.4]{BDGS}
Let $\Omega \subseteq \C^d$ be a $G$-invariant domain. Then 
\begin{enumerate}
    \item[(i)] $\bl \theta(\Omega)$ is a domain, and
    \item[(ii)] $\bl \theta : \Omega \to \bl \theta(\Omega)$ is a proper holomorphic map with $G$ as the group of deck transformations.
\end{enumerate}
\end{prop}
Moreover, if $\bl \theta'$ is another basic polynomial map associated to the group $G,$ then the domain $\bl \theta'(\Omega)$ is biholomorphically equivalent to $\bl \theta(\Omega)$ \cite[p. 12]{BDGS}. The following proposition illustrates a significance of a basic polynomial map.
\begin{prop}\label{representative}
Suppose that $\bl f : \Omega_1 \to \Omega_2$ is a proper holomorphic map with the finite pseudorflection group $G$ as the group of deck transformations. 
Then there exists a unique biholomorphic map $\widehat{\bl f} : \bl \theta(\Omega_1) \to \Omega_2$ such that $\bl f=\widehat{\bl f}\circ \bl \theta$, where $\bl \theta: \Omega_1 \to \bl \theta(\Omega_1)$ is a basic polynomial map associated to the complex reflection group $G.$
\end{prop} 
We use the analytic version of Chevalley-Shephard-Todd theorem to prove this proposition. Let $\m O(\Omega)$ denote the ring of all holomorphic functions on $\Omega.$
\begin{thm}[Analytic CST Theorem]\label{acst}\cite[p. 12]{BDGS}
Let $G$ be a finite complex reflection group and $\Omega \subseteq \mb C^d$ be a $G$-invariant domain. For any $G$-invariant holomorphic function $f$ on $\Omega,$ there exists a unique $\widehat{f} \in \m O\big(\bl \theta(\Omega)\big)$ such that $f=\widehat{f}\circ \bl \theta.$
\end{thm}
A function $\bl f :\Omega_1 \to \Omega_2$ is said to be $G$-invariant if $\bl f(\sigma^{-1} \cdot \bl z) = \bl f(\bl z)$ for all $\sigma \in G, \bl z \in \Omega_1$. Clearly, if $\bl f=(f_1,\ldots,f_d):\Omega_1 \to \Omega_2$ is a $G$-invariant holomorphic map, then each $f_i$ is a $G$-invariant holomorphic map on $\Omega_1$. By analytic CST theorem, we get $f_i=\widehat{f}_i \circ \bl \theta,$ where $\widehat{f}_i \in \m O\big(\bl \theta(\Omega_1)\big).$ That is, $\bl f = \widehat{\bl f} \circ \bl \theta,$ where the function $\widehat{\bl f} = (\widehat{f_1},\ldots,\widehat{f_d}) : \bl \theta(\Omega_1) \to \Omega_2$ is uniquely determined.

\begin{proof}[Proof of Proposition \ref{representative}]
Since the group of deck transformations of $\bl f$ is $G,$ $\bl f$ is $G$-invariant. So there exists a unique holomorphic map $\widehat{\bl f} $ such that $\bl f=\widehat{\bl f}\circ \bl \theta$. In order to show that $\widehat{\bl f} : \bl \theta(\Omega_1) \to \Omega_2$ is biholomorphic, it is enough to prove that the map is bijective. Then from \cite[p. 303, Theorem 15.1.8]{MR2446682}, it follows that $\widehat{\bl f}$ is biholomorphic. 

Being a proper holomorphic map, $\bl f$ is surjective. Since $\bl f= \widehat{\bl f} \circ \bl \theta$, the map $\widehat{\bl f}$ is also surjective. Arguing by contradiction, suppose that $\widehat{\bl f}$ is not injective, then there exist two distinct $\bl z_1, \bl z_2 \in \bl \theta(\Omega_1)$ such that $\widehat{\bl f}(\bl z_1)=\widehat{\bl f}(\bl z_2).$ Suppose that for $\bl w_i \in \Omega_1,$ we have $\bl \theta(\bl w_i)=\bl z_i,$ for $i=1,2.$ That implies 
\Bea
\bl f(\bl w_1)=(\widehat{\bl f} \circ \bl \theta)(\bl w_1) =\widehat{\bl f}(\bl z_1)=\widehat{\bl f}(\bl z_2) = (\widehat{\bl f} \circ \bl \theta)(\bl w_2) = \bl f(\bl w_2).
\Eea
Since $G$ is the group of deck transformations of $\bl f$, there is an element $\sigma \in G$ such that $\sigma(\bl w_1) = \bl w_2.$ Consequently, $G$-invariance of $\bl \theta$ implies that $\bl \theta(\sigma(\bl w_1)) = \bl \theta(\bl w_1)$, that is, $\bl \theta(\bl w_2) = \bl \theta(\bl w_1),$ that is, $\bl z_1=\bl z_2$, which is a contradiction.
\end{proof}
Moreover, the complex jacobian $J_{\bl\theta} = \det\big(\!\!\!\big(\frac{\partial \theta_i}{\partial z_j}\big)\!\!\!\big)_{i,j=1}^d$ has an association with the sign representation of $G.$ A detailed discussion regarding one-dimensional representations of complex reflection groups is provided which consequently establishes the association. 
\subsection{One-dimensional representations of complex reflection groups}
\begin{defn}\rm
A hyperplane $H$ in $\C^d$ is called reflecting if there exists a complex reflection in $G$ acting trivially
on $H$.
\end{defn}
  For a complex reflection $\sigma \in G,$ define $H_{\sigma} := \ker({\rm id} - \sigma).$ By definition, the subspace $H_{\sigma}$ has dimension $d-1$. Clearly, $\sigma$ fixes the hyperplane $H_{\sigma}$ pointwise. Hence each $H_\sigma$ is a reflecting hyperplane.  By definition, $H_\sigma$ is the zero set of a non-zero homogeneous linear polynomial $L_\sigma$ on $\C^d$, determined up to a non-zero constant multiplier, that is, $H_\sigma = \{\bl z\in\C^d: L_\sigma(\bl z) = 0\}$. Moreover, the elements of $G$ acting trivially on a  reflecting hyperplane forms a cyclic subgroup of $G$. 
  
  Let $H_1,\ldots, H_t$ denote the distinct reflecting hyperplanes associated to the group $G$ and  the corresponding cyclic subgroups are $G_1,\ldots, G_t,$ respectively. Suppose $G_i = \langle a_i \rangle$ and the order of each $a_i$ is $m_i$ for $i=1,\ldots,t.$ For every one-dimensional representation $\varrho$ of $G,$ there exists a unique $t$-tuple of non-negative integers $(c_1,\ldots,c_t),$ where $c_i$'s are the least non-negative integers that satisfy the following: \bea\label{ci}\varrho(a_i) =\big( \det(a_i)\big)^{c_i}, \,\, i=1,\ldots,t.\eea The $t$-tuple $(c_1,\ldots,c_t)$ solely depends on the representation $\varrho.$ The character of the one-dimensional representation $\varrho,$ $\chi_\varrho : G \to \mb C^*$ coincides with the representation $\varrho.$ 
The set of polynomials relative to the representation $\varrho$ is given by $$R^G_{\varrho}(\mb C[z_1,\ldots,z_d]) = \{f \in \mb C[z_1,\ldots,z_d] : \sigma(f) = \chi_\varrho(\sigma) f, ~ {\rm for ~~ all~} \sigma \in G\}.$$  The elements of the subspace $R^G_{\varrho}(\mb C[z_1,\ldots,z_d])$ are said to be  $\varrho$-invariant polynomials. Stanley proves a typical property of the elements of $R^G_{\varrho}(\mb C[z_1,\ldots,z_d])$ in \cite[p. 139, Theorem 3.1]{MR460484}.
\begin{lem}\label{gencz}\cite[p. 139, Theorem 3.1]{MR460484}
Suppose that the linear polynomial $\ell_i$ is a defining function of $H_i$ for $i=1,\ldots,t.$ The homogeneous polynomial $\ell_\varrho = \prod_{i=1}^t \ell_i^{c_i }$ is a generator of the module $R^G_{\varrho}(\mb C[z_1,\ldots,z_d])$ over the ring $\mb C[z_1,\ldots,z_d]^G,$ where $c_i$'s are unique non-negative integers as described in Equation \eqref{ci}.
\end{lem}
In particular, the sign representation of a finite complex reflection group $G$, described in Equation \eqref{sign}, is given by ${\rm sgn} (a_i) = \big(\det(a_i)\big)^{m_i-1}$, $i=1,\ldots,t,$ \cite[p. 139, Remark (1)]{MR460484} and it has the following property. 

\begin{cor}\cite[p. 616, Lemma]{MR117285}\label{Jac}
Let $H_1,\ldots, H_t$ denote the distinct reflecting hyperplanes associated to the group $G$ and let $m_1,\ldots, m_t$ be the orders of the corresponding cyclic subgroups $G_1,\ldots, G_t,$ respectively. Suppose that the linear polynomial $\ell_i$ is a defining function of $H_i$ for $i=1,\ldots,t.$ Then for a non-zero constant $c$, 
\Bea
J_{\bl \theta} (\bl z) = c \prod_{i=1}^t \ell_i^{m_i -1 }(\bl z) = \ell_{\rm sgn}(\bl z).
\Eea
Consequently, $J_{\bl \theta}$ is a basis of the module $R^G_{{\rm sgn}}(\mb C[z_1,\ldots,z_d])$ over the ring $\mb C[z_1,\ldots,z_d]^G$.
\end{cor}
Clearly, the character of the sign representation is given by $\chi_{{\rm sgn}} : G \to \C^*$
\bea
\label{mu}\chi_{\rm sgn}(\sigma) = \big(\det(\sigma)\big)^{-1},\,\, \sigma\in G.
\eea 
Now, for any one-dimensional representation $\varrho$ of $G,$ we generalize the notion of $\varrho$-invariance for analytic Hilbert modules. 
\begin{defn} \rm
We recall the following definitions from  \cite{MR1988884} and \cite{MR1028546}.
 \begin{enumerate} 
    \item[1.] A Hilbert space $\mathcal H$ is said to be a \emph{Hilbert module} over an algebra $\m A$ if the map $(f,h) \mapsto T_f(h):=f\cdot h,$ $f\in \m A, h\in \mathcal H,$ defines an algebra homomorphism $f \mapsto T_f$ of $\m A$ into $\m L(\m H).$ 
    \item[2.] A Hilbert module $\mathcal H$ over $\mb C[z_1,\ldots,z_d]$ is said to be an \emph{analytic Hilbert module} if
\begin{enumerate}
\item the Hilbert space $\mathcal H$ consists of holomorphic functions on some bounded domain $\Omega\subseteq \C^d,$
\item $\mb C[z_1,\ldots,z_d]\subseteq\mathcal H$ is dense in $\mathcal H$ and
\item $\mathcal H$ possesses a reproducing kernel on $\Omega.$
\end{enumerate}
The module action in an analytic Hilbert module is given by pointwise multiplication, that is, ${\mathfrak m}_p(h)(\bl z) = p(\bl z) {h}(\bl z),\,h \in \m H,\, \bl z\in \Omega.$ 
\end{enumerate}

For example, a weighted Bergman space on some bounded domain $\Omega \subseteq \mb C^d$ is an analytic Hilbert module over $\mb C[z_1,\ldots,z_d]$ on $\Omega.$
\end{defn}
Let $\m H$ be an analytic Hilbert module  consisting holomorphic functions on $\Omega$, where $\Omega$ is a $G$-invariant domain. We also assume that the reproducing kernel $K$ of $\m H$ is $G$-invariant, that is, 
\Bea
K(\sigma \cdot \bl z, \sigma \cdot \bl w)= K(\bl z, \bl w)\,\, \text{for all}\,\, \sigma \in G.
\Eea
For a one-dimensional representation $\varrho$ of $G,$ recall that the relative invariant subspace of $\m H,$ defined by
\bea\label{insub}
R^G_{\varrho}(\m H) = \{f \in \m H : \sigma(f)  = \chi_\varrho(\sigma) f ~ {\rm for ~~ all~} \sigma \in G \}.
\eea
The elements of the subspace $R^G_{\varrho}(\m H)$ are said to be  $\varrho$-invariant. 
A result analogous to Lemma \ref{gencz}  is proved for  $R^G_{\varrho}(\m H)$ in the next lemma. 

\begin{lem}\label{quo}
Let $\m H\subseteq \m O(\Omega)$ be an analytic Hilbert module and  $f \in R_\varrho^G(\m H).$ Then $\ell_\varrho$ divides $f$ and the quotient $\frac{f}{\ell_\varrho}$ is a $G$-invariant holomorphic function on $\Omega,$ where $\ell_\varrho$ is as in Lemma \ref{gencz}.
\end{lem}
\begin{proof}
For a fixed $1 \leq i \leq t,$ suppose that the reflecting hyperplane $H_i$ is fixed by the cyclic subgroup $G_i$ and the complex reflection $a_i$ generates $G_i.$ Using a linear change of coordinates in $\Omega,$ we consider a new coordinate system $y_1 = \ell_i, y_2=x_2,\ldots,y_d=x_d.$ In this new coordinate system, we have $a_i = {\rm diag}(\omega_i ,1, \ldots, 1),$ where $\omega_i$ is a primitive $m_i$-th root of unity. Then, 
\Bea
f(a_i^{-1} \cdot (y_1,\ldots,y_d)) &=& \chi_\varrho(a_i) f(y_1,\ldots,y_d)\\
f(\omega_i y_1, y_2, \ldots, y_d) &=& \big({\det}(a_i) \big)^{c_i} f(y_1,\ldots,y_d)\\
&=& \omega_i^{c_i} f(y_1,\ldots,y_d).
\Eea
Therefore, $f(y_1,\ldots,y_d)$ is divisible by $y_1^{c_i}.$ Now changing the coordinates, we conclude that $f(x_1,\ldots,x_d)$ is divisible by $\ell_i^{c_i}.$ 

If $\bl z\notin H_i\cap \Omega$, then $\ell_i(\bl z)\neq 0$ and this is true in a neighbourhood of $\bl z.$ On the other hand, if $H_i\cap \Omega$ is non-empty and $\bl z\in H_i\cap \Omega,$ then it follows from  Weierstrass Division Theorem \cite[p. 11]{MR1288523} that $\ell_i$ divides $f$ in a neighbourhood of $\bl z$ since $\ell_i$ is irreducible. As the holomorphic functions which are obtained locally, patch to give a global holomorphic function, since both $f$ and $\ell_i$ are defined on all of $\Omega$, it follows that $\ell_i$ divides $f$ in $\mathcal O(\Omega).$  Repetitive application of this argument gives us that the quotient function $\frac{f}{\ell_i^{c_i}}$ is in $\m O(\Omega)$. These steps can be repeated for all distinct hyperplanes associated to the group $G$. Thus one gets that $\ell_\varrho$ divides $f$ and the quotient function is holomorphic. Moreover, $\ell_\varrho$ is $\varrho$-invariant from \cite[p. 138, Theorem 2.3]{MR460484}. The quotient of a $\varrho$-invariant function by a $\varrho$-invariant function is clearly $G$-invariant. Hence $\frac{f}{\ell_\varrho}$ is $G$-invariant. 
\end{proof}
\begin{rem}\rm\label{expression} Let $f \in R_\varrho^G(\m H).$ Combining Theorem \ref{acst} and Lemma \ref{quo}, we write $\frac{f}{\ell_\varrho} =  \widehat{f} \circ \bl \theta$ for a unique $\widehat{f} \in \mathcal O\big(\bl \theta(\Omega)\big).$ 
  That is, $f =(\widehat{f} \circ \bl \theta)\ell_\varrho.$ 
  In other words, $R^G_{\varrho}(\m H) \subseteq \{ (\widehat{f} \circ \bl \theta)\ell_\varrho\mid \widehat{f} \in \mathcal O\big(\bl \theta(\Omega)\big)\}.$ 
\end{rem}

\subsection{Orthogonal projections}
Each one-dimensional representation $\varrho$ of $G$ induces the unique linear operator $\mb P_\varrho : \m H \to \m H$ given by \bea\label{projmu}
\mb P_\varrho \phi = \frac{1}{|G|}\sum_{\sigma \in G} \chi_\varrho(\sigma^{-1}) ~ \phi \circ \sigma^{-1}, \, \, \, \phi \in \m H,
\eea
 where $\chi_\varrho$ is the character of the representation $\varrho$ and $|G|$ denotes the order of the group $G$.
\begin{lem}\label{ortho}
Let $G$ be a finite pseudoreflecion group and $\Omega$ be a $G$-invariant domain in $\C^d.$ If $\m H \subseteq \m O(\Omega)$ is an analytic Hilbert module with the $G$-invariant reproducing kernel, the operator $\mb P_\varrho: \m H \to \m H$ is the orthogonal projection onto the subspace $R^G_{\varrho}(\m H).$
\end{lem}
\begin{proof}
The reproducing kernel of the analytic Hilbert module $\m H$ is $G$-invariant. So $\mb P_\varrho$ is an orthogonal projection on $\m H$ by \cite[p. 15]{BDGS}. It remains to show that  $\mb P_\varrho \m H =  R^G_{\varrho}(\m H).$
 For any $\tau \in G,$ $\phi \in \mb P_\varrho \m H,$ \Bea \tau(\phi) = \tau(\mb P_\varrho \phi) &=& \frac{1}{|G|} \sum_{\sigma \in G} \chi_\varrho(\sigma^{-1}) ~ \phi \circ \sigma^{-1}\tau^{-1} \\ &=& \frac{1}{|G|} \sum_{\eta \in G} \chi_\varrho(\eta^{-1}\tau ) ~ \phi \circ \eta^{-1} \,\, (\text{~taking~}\eta = \tau\sigma)\\ &=& \chi_\varrho(\tau) ~ \phi. \Eea
Hence $\phi$ is in $R^G_{\varrho}(\m H).$

To prove $R^G_{\varrho}(\m H) \subseteq \mb P_\varrho \m H,$ consider $\phi \in R^G_{\varrho}(\m H).$ Then we have
\Bea \mb P_\varrho (\phi) = \frac{1}{|G|} \sum_{\sigma \in G}  \chi_\varrho(\sigma^{-1}) ~ \phi \circ \sigma^{-1} &=& \frac{1}{|G|} \sum_{\sigma \in G}  \chi_\varrho(\sigma^{-1}) \chi_\varrho(\sigma) ~ \phi  \\ &=& \frac{1}{|G|} \sum_{\sigma \in G} \phi = \phi. \Eea This completes the proof.
\end{proof}
The group action on the set of all complex-valued function in Introduction is known as the left regular representation. The left regular representation of $G$ admits the isotypic decomposition $\oplus_{\varrho \in \widehat{G}} n_\varrho \varrho,$ where the positive integer $n_\varrho = \chi_{\varrho}({\rm id})$ is the multiplicity of $\varrho$ in the left regular representation of $G$ and $\widehat{G}$ is the set of equivalence classes of irreducible representations of $G.$ The linear map $\mb P_\varrho$ is the orthogonal projection onto the isotypic component associated to the irreducible representation $\varrho$ in the decomposition of the left regular representation on $\m H.$ 
\begin{thm}\cite{BDGS}\label{lula}
An analytic Hilbert module $\m H$ on a $G$-space $\Omega$ with $G$-invariant kernel admits an orthogonal decomposition $\m H = \oplus_{\varrho \in \widehat{G}} \mb P_\varrho \m H$ and each $\mb P_\varrho \m H$ is isometrically isomorphic to a reproducing kernel Hilbert space $\m H_\varrho \subseteq \m O(\bl \theta(\Omega)) \otimes \mb C^{n_\varrho^2},$ where $n_\varrho = \chi_\varrho({\rm id}).$
\end{thm}

\section{On the weighted Bergman modules}\label{mainsection}
In this section, we prove Theorem \ref{firstmain}  and Theorem \ref{maint}. We start by recalling weighted Bergman spaces. Let $\Omega$ be a bounded domain in $\mb C^d.$ For a continuous weight function $\omega: \Omega \to (0,\infty),$ the weighted Bergman space $\mb A^2_\omega(\Omega)$ is defined by
$$\mb A^2_\omega(\Omega) = \{f \in \m O(\Omega) : \int_\Omega |f(\bl z)|^2 \omega(\bl z) dV(\bl z) < \infty\},$$
where $dV$ is the normalized Lebesgue measure on $\Omega.$ Clearly, the weighted Bergman space $\mb A^2_\omega(\Omega)$ is an analytic Hilbert module. 


Let $G$ be a finite complex reflection group and $\Omega$ be a $G$-space. We consider a weight function of the form $\omega =\widetilde{\omega}\circ \bl \theta$ for a continuous map $\widetilde{\omega} : \bl \theta(\Omega) \to (0,\infty),$ where $\bl \theta$ is a basic polynomial map associated to the group $G.$ Let $\mathcal U_d$ denote the group of unitary operators on $\C^d.$ Since any finite group $G$ generated by complex reflections on $\C^d$ is a subgroup of $\mathcal U_d$, it can be shown that the reproducing kernel of the weighted Bergman space $\mb A^2_\omega(\Omega)$ is $G$-invariant as the weight function $\omega$ is $G$-invariant. Specializing   Lemma \ref{ortho} to the weighted Bergman module $\mb A^2_\omega(\Omega),$ we obtain the following result.
\begin{lem}\label{lalala}
Let $G$ be a finite complex reflection group and 
$\Omega \subseteq \mb C^d$ be a $G$-invariant domain. 
Then  $\mb P_\varrho: \mb A^2_\omega(\Omega) \to \mb A^2_\omega(\Omega)$ is an orthogonal projection and $\mb P_\varrho\big (\mb A^2_\omega(\Omega)\big) = R^G_{\varrho}\big( \mb A^2_\omega(\Omega)\big).$
\end{lem}
For each one-dimensional representation $\varrho \in \widehat{G},$ note that the function $ \frac{|\ell_\varrho(\bl z)|^2}{|J_{\bl \theta}(\bl z)|^2}$ is a positive-valued continuous function on $\Omega \setminus \cup_{i=1}^t H_i,$ where the hyperplanes $H_i$'s are as in Corollary \ref{Jac}. Since $\ell_\varrho$ is $\varrho$-invariant, there exists a unique positive integer, say $m_\varrho,$ such that $\ell_\varrho^{m_\varrho}$ is $G$-invariant. Then $\ell_\varrho^{m_\varrho}(\bl z) = f_\varrho(\bl \theta(\bl z))$ for a unique polynomial $f_\varrho$. The function $\omega_\varrho : \bl \theta(\Omega) \setminus \bl \theta(\cup_{i=1}^t H_i) \to (0,\infty)$ defined by \bea\label{weight}\omega_\varrho(\bl u) = \frac{|f_\varrho(\bl u)|^\frac{2}{m_\varrho}}{|f_\sgn(\bl u)|^\frac{2}{m_{\rm sgn}}} \widetilde{\omega}(\bl u), \,\, \bl u \in \bl \theta(\Omega) \setminus \bl \theta(\cup_{i=1}^t H_i)\eea is continuous.
With each one-dimensional representation $\varrho \in \widehat{G},$ we associate the linear map $\Gamma_\varrho : \mb A_{\omega_\varrho}^2(\bl \theta(\Omega)) \to \mb A^2_\omega(\Omega)$ defined by \Bea \Gamma_\varrho \psi = \frac{1}{\sqrt{|G|}}  (\psi \circ \bl \theta)\ell_\varrho .\Eea 
Note that
\Bea \norm{\Gamma_\varrho \psi}^2 &=& \frac{1}{|G|} \int_\Omega |\psi \circ \bl \theta(\bl z)|^2 |\ell_\varrho(\bl z)|^2 \omega(\bl z) dV(\bl z) \\ &=& \frac{1}{|G|} \int_{\Omega \setminus \cup_{i=1}^t H_i} |\psi \circ \bl \theta(\bl z)|^2 \frac{|f_\varrho(\bl \theta(\bl z))|^\frac{2}{m_\varrho}}{|f_{\rm sgn}(\bl \theta(\bl z))|^\frac{2}{m_{\rm sgn}}}  \widetilde{\omega}(\bl \theta(\bl z)) |J_{\bl \theta}(\bl z)|^2 dV(\bl z) \\ &=& \int_{\bl \theta(\Omega) \setminus \bl \theta(\cup_{i=1}^t H_i)} |\psi(\bl z)|^2 \omega_\varrho(\bl z) dV(\bl z) = \norm{\psi}^2, \Eea the last equality follows since the set $\bl \theta(\Omega) \setminus \bl \theta(\cup_{i=1}^t H_i)$ is of Lebesgue measure $0.$ So $\Gamma_\varrho$ is an isometry. Therefore, the range of $\Gamma_\varrho$ is closed and $\Gamma_\varrho\big(\mb A_{\omega_\varrho}^2(\bl \theta(\Omega))\big)$ is a reproducing kernel Hilbert space. Moreover, we show in the next lemma that it coincides with the relative invariant subspace associated to the representation $\varrho.$
\begin{lem}\label{equiv}
Let $G$ be a finite pseudoreflecion group and $\Omega$ be a $G$-invariant domain in $\C^d.$ Suppose that $\bl \theta :\Omega \to \bl \theta(\Omega)$ is a basic polynomial map associated to the group $G$. Then $\Gamma_\varrho : \mb A_{\omega_\varrho}^2(\bl \theta(\Omega)) \to R^G_\varrho\big(\mb A^2_\omega(\Omega)\big)$ is unitary.
\end{lem}
\begin{proof}
Since $\Gamma_\varrho$ is an isometry, it is enough to prove that $\Gamma_\varrho \big( \mb A_{\omega_\varrho}^2(\bl \theta(\Omega)) \big) = R^G_\varrho\big(\mb A^2_\omega(\Omega)\big).$
Suppose that $h \in \mb A_{\omega_\varrho}^2(\bl \theta(\Omega)).$ Then for every $\sigma \in G$,
\Bea \sigma(\Gamma_\varrho h)= \frac{1}{\sqrt{|G|}} \sigma\big((h \circ \bl \theta)\ell_\varrho\big) &=&  \frac{1}{\sqrt{|G|}} \sigma(\ell_\varrho) \sigma(h \circ \bl \theta)\\ &=& \frac{1}{\sqrt{|G|}} \chi_\varrho (\sigma) \ell_\rho (h \circ \bl \theta) = \chi_\varrho(\sigma) \Gamma_\varrho h . \Eea Hence $\Gamma_\varrho h$ is a $\varrho$-invariant element. 

On the other hand, consider $h \in R^G_\varrho \big(\mb A^2_\omega(\Omega)\big).$ By Remark \ref{expression}, we write $h = \ell_\varrho (\widehat{h} \circ \bl \theta),$ for $\widehat{h} \in \mathcal O\big(\bl \theta(\Omega)\big).$ Clearly, 
\Bea
\norm{h}^2 = \int_{\Omega} |h(\bl z)|^2 \omega(\bl z) dV(\bl z) &=& \int_{\Omega} |\ell_\varrho(\bl z)|^2 |\widehat{h} \circ \bl \theta(\bl z)|^2 \omega(\bl z) dV(\bl z) \\ &=& |G| \int_{\bl \theta(\Omega)} |\widehat{h}(\bl z)|^2\omega_\varrho(\bl z) dV(\bl z).
\Eea
This shows that $\widehat{h} \in \mb A_{\omega_\varrho}^2\big(\bl \theta(\Omega)\big)$ and $\Gamma_\varrho(\sqrt{|G|}~ \widehat{h}) = h.$ Therefore, $h \in \Gamma_\varrho \big(\mb A_{\omega_\varrho}^2\big(\bl \theta(\Omega)\big)\big).$
\end{proof}
\begin{rem}\rm
From Remark \ref{expression}, we have $R^G_\varrho\big(\mb A^2_\omega(\Omega)\big) \subseteq \{\ell_\varrho~(\widehat{f} \circ \bl \theta)\mid \widehat{f} \in \mathcal O\big(\bl \theta(\Omega)\big)\}.$ Whereas Lemma \ref{equiv} asserts that $$R^G_\varrho \big( \mb A^2_\omega(\Omega)\big) = \{\ell_\varrho~(\widehat{f} \circ \bl \theta) \mid \widehat{f} \in \mb A_{\omega_\varrho}^2\big(\bl \theta(\Omega)\big)\}.$$ We refer $\ell_\varrho$ by \emph{generating polynomial} of $R^G_\varrho\big(\mb A^2_\omega(\Omega)\big).$
\end{rem}
Now we have all the ingredients to prove Theorem \ref{firstmain}.
\begin{proof}[\bf{Proof of Theorem \ref{firstmain}}]
The first part follows from Lemma \ref{equiv}.

In Lemma \ref{ortho}, we show that $\mb P_\varrho$ is the orthogonal projection onto the subspace $R^G_\varrho \big(\mb A^2_\omega(\Omega)\big) = \Gamma_\varrho \big(\mb A_{\omega_\varrho}^2\big(\bl \theta(\Omega)\big)\big).$ Then the reproducing kernel $\m B_\varrho$ of $\Gamma_{\varrho} \big(\mb A_{\omega_\varrho}^2\big(\bl \theta(\Omega)\big)\big)$ is given by 
\bea\label{sec}
\m B_\varrho(\bl z, \bl w) = \big( \mb P_\varrho (\m B_\omega)_{\bl w} \big) (\bl z) 
= \frac{1}{|G|} \sum_{\sigma \in G}\chi_\varrho(\sigma^{-1})  \m B_\omega(\sigma^{-1} \cdot \bl z, \bl w ).
\eea
An analogous calculation as in \cite[p. 6]{ghosh2020multiplication} yields the following formula for the reproducing kernel  $\m B_\varrho$ of $\Gamma_\varrho\big(\mb A_{\omega_\varrho}^2(\bl \theta(\Omega))\big)$:
\bea\label{firstmu}
\m B_\varrho(\bl z,\bl w)= \frac{1}{|G|}\ell_\varrho(\bl z)\m B_{\omega_\varrho}\big(\bl \theta(\bl z),\bl \theta(\bl w)\big)\ov{\ell_\varrho(\bl w)}  \mbox{~for~} \bl z, \bl w\in\Omega,
\eea
where $\m B_{\omega_\varrho}$ is the kernel of $\mb A_{\omega_\varrho}^2(\bl \theta(\Omega))$. Combining Equation \eqref{sec} and Equation \eqref{firstmu}, we have the desired  result.
\end{proof}

\begin{rem}\rm
 For a fixed $\bl w \in \Omega$, as a function of $\bl z,$ we write $\mathcal B_\varrho(\bl z, \bl w) = \big( \mb P_\varrho (\m B_\omega)_{\bl w} \big) (\bl z) = \ell_\varrho(\bl z) \big((\m B_\omega)_{\bl w}^\varrho \circ \bl \theta(\bl z)\big)$ for some $(\m B_\omega)_{\bl w}^\varrho \in \mathcal O(\bl \theta(\Omega))$. Similarly we have, $\mathcal B_\varrho (\bl z, \bl w)= \big(\overline{ \mb P_\varrho (\m B_\omega)_{\bl z}} \big) (\bl w) = \overline{\ell_\varrho(\bl w)} \big( \overline{(\m B_\omega)_{\bl z}^\varrho \circ \bl \theta(\bl w)\big)}$ for some $(\m B_\omega)_{\bl z}^\varrho \in \mathcal O(\bl \theta(\Omega)),$ when $\bl z$ is fixed but arbitrary. Since the variables $\bl z$ and $\bl w$ are independent of each other,  $\mathcal B_\varrho (\bl z, \bl w)$ divisible by $\ell_\varrho(\bl z) \overline{\ell_\varrho(\bl w)}$ for every $\bl z, \bl w \in \Omega.$ Therefore, the right side of Equation \eqref{formu} is well-defined, even if $\bl z$ or $\bl w$ belongs to $N_\varrho=\{\bl z \in \Omega:  \ell_\varrho (\bl z) =0\}.$
\end{rem}
 Additionally, if $G$ is abelian, we have the following result.
\begin{prop}
Suppose that $G$ is an abelian group, then 
\Bea
\mb A^2_\omega(\Omega) \cong \displaystyle\oplus_{\varrho \in \widehat{G}} \mb A^2_{\omega_\varrho}(\bl \theta(\Omega)),
\Eea
and 
\Bea
\m B_\omega(\bl z,\bl w) = \frac{1}{|G|} \sum_{\varrho \in \widehat{G}} \ell_\varrho(\bl z) \m B_{\omega_\varrho}(\bl \theta(\bl z),\bl \theta(\bl w))\overline{\ell_\varrho(\bl w)}, \, \, \bl z, \bl w \in \Omega,
\Eea
 where $\m B_\omega$ and $\m B_{\omega_\varrho}$ are the reproducing kernels of $\mb A_\omega(\Omega)$ and $\mb A^2_{\omega_\varrho}(\bl \theta(\Omega)),$ respectively,  and the polynomial $\ell_\varrho$ is as described in Lemma \ref{gencz} for $\varrho \in \widehat{G}.$
\end{prop}

\begin{proof}
Since $G$ is a finite abelian group, each irreducible representation of $G$ is one-dimensional. The isotypic decomposition of the left regular representation of $G$ on $\mb A^2_\omega(\Omega)$  yields the orthogonal decomposition $\mb A^2_\omega(\Omega) = \oplus_{\varrho \in \widehat{G}} \mb P_\varrho\big(\mb A^2_\omega(\Omega)\big)$ (cf. Theorem \ref{lula}). Moreover, $\mb P_\varrho$ is the orthogonal projection onto the subspace $R^G_\varrho\big(\mb A^2_\omega(\Omega)\big)$ (cf. Lemma \ref{lalala}) which is isometrically isomorphic to $\mb A^2_{\omega_\varrho}(\bl \theta(\Omega))$ (cf. Lemma \ref{equiv}). Thus, the first part follows.

For every fixed $\bl w \in \Omega,$ the function $ \m B_\omega(\cdot,\bl w) \in \mb A^2_\omega (\Omega)$ and it admits the decomposition $\m B_\omega(\bl z,\bl w) = \sum_{\varrho \in \widehat{G}} \mb P_\varrho \m B_\omega(\bl z,\bl w).$  Then Equation \eqref{firstmu} proves the second part.
\end{proof}

We are now in a position to prove Theorem \ref{maint}.
\begin{proof}[\bf{Proof of Theorem \ref{maint}}]
First note that for a fixed $\bl w \in \Omega,$ \Bea(\Gamma_\varrho(\m B_{\omega_\varrho})_{\bl \theta(\bl w)})(\bl z) &=& \frac{1}{\sqrt{|G|}} \ell_\varrho(\bl z)  \m B_{\omega_\varrho}(\bl \theta(\bl z),\bl \theta(\bl w)) \\ &=& \frac{1}{\sqrt{|G|}} \frac{1}{\ov{\ell_\varrho(\bl w)}} \displaystyle\sum_{\sigma \in G} \chi_\varrho(\sigma^{-1}) \m B_\omega(\sigma^{-1} \cdot \bl z, \bl w).\Eea
Moreover, the $\Gamma_\varrho$ map can be extended to $L^2_{\omega_\varrho}(\bl \theta(\Omega))$ isometrically. Therefore, for a fixed but arbitrary $\bl w \in \Omega$ we have the following: $\phi \in L^2_{\omega_\varrho}(\bl \theta(\Omega)),$
\Bea(P_{\bl \theta(\Omega)}^{\omega_\varrho} \phi)( \bl \theta (\bl w)) = \inner{\phi}{(\m B_{\omega_\varrho})_{\bl \theta(\bl w)}} &=&\inner{\Gamma_\varrho \phi}{\Gamma_\varrho (\m B_{\omega_\varrho})_{\bl \theta(\bl w)}} \\ &=& \frac{1}{|G|} \inner{\ell_\varrho(\phi\circ \bl \theta )}{\frac{1}{\ov{\ell_\varrho(\bl w)}} \displaystyle\sum_{\sigma \in G} \chi_\varrho(\sigma^{-1}) (\m B_\omega)_{\sigma \cdot \bl w} } \\ &=& \frac{1}{|G|} \frac{1}{\ell_\varrho(\bl w)} \displaystyle\sum_{\sigma \in G} \chi_\varrho(\sigma) \inner{\ell_\varrho(\phi\circ \bl \theta )}{(\m B_\omega)_{\sigma \cdot \bl w}} \\ &=& \frac{1}{|G|} \frac{1}{\ell_\varrho(\bl w)} \inner{\displaystyle\sum_{\sigma \in G} \chi_\varrho(\sigma) \sigma^{-1}(\ell_\varrho)(\phi\circ \bl \theta )}{(\m B_\omega)_{\bl w}} \\ &=& \frac{1}{\ell_\varrho(\bl w)} \inner{\ell_\varrho(\phi\circ \bl \theta )}{(\m B_\omega)_{ \bl w}} \\ &=& \frac{1}{\ell_\varrho(\bl w)} P_\Omega^\omega (\ell_\varrho(\phi\circ \bl \theta ))(\bl w).\Eea
Hence the result follows.
\end{proof}

{\bf Sign representation.} For the sign representation of $G$, we observe that the relative invariant subspace $R_{\rm sgn}^G(\mb A^2_\omega(\Omega))$ is isometrically isomorphic to the Bergman space $\mb A^2_{\widetilde{\omega}}(\bl \theta(\Omega)),$ where $\omega = \widetilde{\omega} \circ \bl \theta.$ In particular, for the weight function $\omega \equiv 1,$ we have that the Bergman space on $\bl \theta(\Omega)$ is isometrically isomorphic to a subspace $R_{\rm sgn}^G(\mb A^2(\Omega))$ of $\mb A^2(\Omega).$ Subsequently, we get the following results.

We recall from Equation \eqref{mu} that the character of the sign representation is given by $$\chi_{\rm sgn}(\sigma^{-1}) = \det(\sigma)$$ for $\sigma \in G$. Also from Corollary \ref{Jac}, we have that $\ell_{\rm sgn} = J_{\bl \theta}.$ Therefore, the following corollary is an immediate consequence of Theorem \ref{firstmain}. 

\begin{cor}\label{parti}
The weighted Bergman kernel $\m B_{\widetilde{\omega}}$ of $\mb A^2_{\widetilde{\omega}}(\bl \theta(\Omega))$ is given by the following formula:
\bea\label{thir}\m B_{\widetilde{\omega}}\big(\bl \theta(\bl z), \bl \theta(\bl w)\big) = \frac{1}{J_{\bl \theta}(\bl z) \ov{J_{\bl \theta}(\bl w)}} \displaystyle\sum_{\sigma \in G} \det(\sigma) \m B_\omega(\sigma^{-1} \cdot \bl z, \bl w)  \,\,\,\,\,\, \text{for}\,\,  \bl z, \bl w \in \Omega,
\eea
where $\m B_\omega$ is the reproducing kernel of $\mb A^2_\omega(\Omega)$ and $J_{\bl \theta}$ is the determinant of the complex jacobian matrix of the basic polynomial map $\bl \theta.$
\end{cor}

Now we state a formula involving the weighted Bergman projections $P_\Omega^\omega: L^2_\omega(\Omega) \to \mb A^2_\omega(\Omega)$ and $P_{\bl \theta(\Omega)}^{\widetilde{\omega}} : L^2_{\widetilde{\omega}}\big(\bl \theta(\Omega)\big) \to \mb A^2_{\widetilde{\omega}}\big(\bl \theta(\Omega)\big)$ which follows immediately  from Theorem \ref{maint}. This is a generalization of \cite[p. 167, Theorem 1]{MR610182} to the weighted Bergman projections. However, the choice of proper holomorphic map in \cite[p. 167, Theorem 1]{MR610182} is restricted here to a basic polynomial map associated to some finite complex reflection group.
\begin{cor}
The weighted Bergman projections $P_\Omega^\omega$ and $P_{\bl \theta(\Omega)}^{\widetilde{\omega}}$ are related to
\Bea
P_\Omega^\omega\big(J_{\bl \theta} ~ (\phi \circ \bl \theta) \big)= J_{\bl \theta}\big((P_{\bl \theta(\Omega)}^{\widetilde{\omega}} \phi) \circ \bl \theta\big) , \,\, \phi \in   L^2_{\widetilde{\omega}}\big(\bl \theta(\Omega)\big),
\Eea
where $\omega = \widetilde{\omega} \circ \bl \theta.$
\end{cor}

\section{Proper Holomorphic maps and Bergman Kernels}
In this section, we prove a transformation formula for weighted Bergman kernels under   a proper holomorphic map whose  group of deck transformations is either a finite complex reflection of group or a conjugate to a finite complex reflection group.

Suppose that $\bl f : \Omega_1 \to \Omega_2$ is a proper holomorphic map with the finite pseudorflection group $G$ as the group of deck transformations and $\omega:\Omega_1 \to (0,\infty)$ is a continuous function of the form $\omega= \widetilde{\omega}\circ \bl f$ for a continuous function $\widetilde{\omega}:\Omega_2 \to (0,\infty)$. A transformation rule for the weighted Bergman kernels of $\mb A^2_\omega(\Omega_1)$ and $\mb A^2_{\widetilde{\omega}}(\Omega_2)$ under the proper holomorphic map $\bl f$ is established in next theorem.
\begin{thm}
The reproducing kernels $\m B_\omega$  of $\mb A^2_\omega(\Omega_1)$ and $\m B_{\widetilde{\omega}}$ of $\mb A^2_{\widetilde{\omega}}(\Omega_2)$ transform according to
 \Bea\m B_{\widetilde{\omega}}\big(\bl f(\bl z), \bl f(\bl w)\big) = \frac{1}{J_{\bl f}(\bl z) \ov{J_{\bl f}(\bl w)}} \displaystyle\sum_{\sigma \in G} \det(\sigma) ~\m B_\omega(\sigma^{-1} \cdot \bl z, \bl w)  \,\,\,\,\,\, \text{for}\,\,  \bl z, \bl w \in \Omega_1.\Eea
\end{thm}
\begin{proof}
It follows from Proposition \ref{representative} that there exists a unique biholomorphic map $\widehat{\bl f} : \bl \theta(\Omega_1) \to \Omega_2$ such that $\bl f=\widehat{\bl f}\circ \bl \theta.$ We write $\omega'=\widetilde{\omega}\circ \widehat{\bl f}$ and thus $\omega = \omega' \circ \bl \theta.$

Under the biholomorphic map $\widehat{\bl f} : \bl \theta(\Omega_1) \to \Omega_2,$ the weighted Bergman kernels $\m B_{\omega'}$ of $\mb A^2_{\omega'}(\bl \theta(\Omega_1))$ and $\m B_{\widetilde{\omega}}$ of $\mb A^2_{\widetilde{\omega}}(\Omega_2)$ are related as:
\bea\label{baje}\m B_{\widetilde{\omega}}\big(\widehat{\bl f}(\bl \theta(\bl z)), \widehat{\bl f}(\bl \theta(\bl w))\big) = \frac{1}{J_{\widehat{\bl f}}(\bl \theta(\bl z)) \ov{J_{\widehat{\bl f}}(\bl \theta(\bl w))}} \m B_{\omega'}(\bl \theta(\bl z), \bl \theta(\bl w)),\eea see \cite{MR2484092}. Hence we get the result combining Equation \eqref{baje} and Corollary \ref{parti}.
\end{proof}
 
 In particular, for $\omega \equiv 1$ we get the following the transformation formula for the Bergman kernels of $\Omega_1$ and $\Omega_2$ which overlaps with the Bell's transformation formula described in \cite[p. 687, Theorem 1]{MR645338}. We emphasize that our transformation formula works for the critical points of $\bl f$ as well.
\begin{cor}\label{main}
Let  $\bl f : \Omega_1 \to \Omega_2$ be a proper holomorphic map with a finite complex reflection group $G$ as the group of deck transformations. Then 
\bea\label{bajee} \m B_2\big(\bl f(\bl z), \bl f(\bl w)\big) = \frac{1}{J_{\bl f}(\bl z)\overline{J_{\bl f}(\bl w)}}
 \displaystyle\sum_{\sigma \in G} \det(\sigma) \m B_1(\sigma^{-1} \cdot \bl z, \bl w),\,\, \bl z, \bl w \in \Omega_1,\eea
where $\m B_1$ and $\m B_2$ denote the Bergman kernels of $\Omega_1$ and $\Omega_2,$ respectively.
\end{cor}
The Bergman kernel of a domain is \emph{rational} if it is a rational function of the coordinates. If the Bergman kernel $\m B_1 = \frac{q}{p}$ is rational, then from Equation \eqref{bajee} it is clear that $\m B_2$ is also rational if $\bl f$ is a basic polynomial map associated to some complex reflection group. Since the denominator of the sum is given by some polynomial $\prod_\sigma p(\sigma^{-1}\cdot \bl z,\overbar{\bl w})=\prod_\sigma p( \bl z,\sigma^{-1}\cdot\overbar{\bl w})$ ($\m B_1$ is a $G$-invariant kernel) which is $G$-invariant in both $\bl z$ and $\overbar{\bl w}$ and thus by Chevalley-Shephard-Todd theorem the denominator is a polynomial in both $\bl f(\bl z)$ and $\overbar{\bl f(\bl w)}.$ For numerator of the sum, we observe that it can be written as $N(\bl z,\bl w) = \prod_{\tau \in G} q(\tau^{-1}\cdot \bl z,\bl w) \sum_{\sigma \in G} \det (\sigma) \frac{p(\sigma^{-1}\cdot \bl z,\bl w)}{q(\sigma^{-1}\cdot \bl z,\bl w)}.$ Clearly, 
\Bea
N(\sigma_0 \cdot \bl z,\bl w) &=& \prod_{\tau \in G} q(\tau^{-1}\cdot \bl z,\bl w) \sum_{\sigma \in G} \det (\sigma) \frac{p(\sigma^{-1}\sigma_0\cdot \bl z,\bl w)}{q(\sigma^{-1}\sigma_0\cdot \bl z,\bl w)}\\&=& \prod_{\tau \in G} q(\tau^{-1}\cdot \bl z,\bl w) \sum_{\gamma \in G} \det(\sigma_0) \det (\gamma) \frac{p(\gamma^{-1}\cdot \bl z,\bl w)}{q(\gamma^{-1}\cdot \bl z,\bl w)} \\&=& \chi_{\rm sgn}(\sigma_0^{-1}) N(\bl z,\bl w).
\Eea Using $G$-invariance of the Bergman kernel, we get $N(\bl z,\sigma_0 \cdot \bl w) = \chi_{\rm sgn}(\sigma_0) N(\bl z,\bl w).$ Hence $N(\bl z,\bl w)$ is divisible by 
 $\frac{1}{J_{\bl f}(\bl z)\overline{J_{\bl f}(\bl w)}}$ and the residue is a $G$-invariant polynomial and then using Chevalley-Shephard-Todd theorem we conclude the following.
\begin{cor}\label{ratcat}
Suppose that $\Omega_1$ is a $G$-invariant domain in $\mb C^d$ and the Bergman kernel of $\Omega_1$ is rational. Then the Bergman kernel of the domain $\bl \theta(\Omega_1)$ is also rational for a basic polynomial map $\bl \theta$ associated to $G.$ 
\end{cor}
\subsection{Groups Conjugate to Complex reflection groups}
Suppose that $\Omega_1 \text{~and~} \Omega_2$ are two domains in $\mb C^d$ and $G \subseteq {\rm Aut}(\Omega_1)$ is a finite complex reflection group.  The group \bea\label{conju}\widetilde{G}=\bl \Psi^{-1} G \bl \Psi\eea is said to be a conjugate to the complex reflection group $G$ by an automorphism $\bl \Psi \in {\rm Aut}(\Omega_1).$ Let $\bl F: \Omega_1 \to \Omega_2$ be a proper holomorphic map with the group of deck transformations $\widetilde{G}.$ Equivalently, \bea\label{facaut}
\bl F^{-1}\bl F(\bl z)=\bigcup_{\sigma\in \widetilde{G}}\{\sigma ( \bl z)\} \,\, \text{~for~} \bl z\in \Omega_1.
\eea
The proper holomorphic map $\bl F$ satisfying Equation \eqref{facaut} is referred as \emph{factored by automorphisms} $\widetilde{G}$ in \cite{MR807258,MR1131852}. We obtain a characterization for such proper holomorphic maps in the following proposition.
\begin{prop}\label{express}
Suppose that $\bl F: \Omega_1 \to \Omega_2$ is a proper holomorphic map between two bounded domains in $\mb C^d$ and $\widetilde{G}\subseteq {\rm Aut}(\Omega_1)$ is a conjugate to a complex reflection group $G$ by the automorphism $\bl \Psi \in {\rm Aut}(\Omega_1)$. Then $\bl F$ is factored by $\widetilde{G}$ if and only if  $\bl F= \bl \Phi \circ \bl \theta \circ \bl \Psi$, where $\bl \theta$ is a basic polynomial map associated to the group $G$ and $\bl \Phi$ is a biholomorphic map from $\bl \theta(\Omega_1)$ to $\Omega_2$.
\end{prop}

\begin{proof}
Suppose that the proper holomorphic map $\bl F$ is factored by the group $\bl \Psi^{-1} G \bl \Psi$. Then
\Bea
\bl F^{-1}\bl F(\bl z)=\bigcup_{\sigma \in G} \{ (\bl \Psi^{-1} \circ \sigma \circ \bl \Psi) (\bl z)\} \text{~for all~}\bl  z\in \Omega_1.
\Eea Consider the map $\bl f = \bl F \circ \bl \Psi^{-1}.$ For every $\bl z \in \Omega_1,$ $\bl f^{-1}\bl f(\bl z)=\bigcup_{\sigma \in G}  \{\sigma (\bl z)\},$ that is, the proper holomorphic map $\bl f : \Omega_1 \to \Omega_2$ is factored by the finite complex reflection group $G.$ From Proposition \ref{representative}, we get that $\bl f = \bl \Phi \circ \bl \theta$, where $\bl \Phi$ is a biholomorphic map from $\bl \theta(\Omega_1)$ to $\Omega_2$ and $\bl \theta$ is a basic polynomial map associated to the group $G.$ Therefore, $\bl F$ can be written in the desired way. 

Conversely, assume that we can express $\bl F= \bl \Phi \circ \bl \theta \circ \bl \Psi$, where $\bl \Phi$ is a biholomorphic map from $\bl \theta(\Omega_1)$ to $\Omega_2$, $\bl \theta$ is a basic polynomial map associated to $G$ and $\bl \Psi \in {\rm Aut}(\Omega_1)$. Note that $\bl \Phi \circ \bl \theta$ is factored by the group $G.$ Since $\bl \Phi \circ \bl \theta(\bl z) = \bl F \circ \bl \Psi^{-1}(\bl z)$ for all $\bl z \in \Omega_1,$ the result follows.
\end{proof}
\begin{thm}\label{main2}
Suppose that $\Omega_i,$ for $i=1,2$ are two bounded domains in $\mb C^d$ and $\bl F: \Omega_1 \to \Omega_2$ is a proper holomorphic map which is factored by $\widetilde{G} \subseteq {\rm Aut}(\Omega_1)$, where $\widetilde{G}$ is as in Equation \eqref{conju}. Then $\m B_2$ can be expressed in terms of $\m B_1$ by the following formula: 
\bea  \label{maineq2}
\m B_2\big(\bl F(\bl z), \bl F(\bl w)\big) =\frac{1}{J_{\bl F}(\bl z) \ov{J_{\bl F}(\bl w)}} \displaystyle\sum_{\sigma \in G} J_{\bl \Psi_\sigma}(\bl z) \m B_1\big(\bl \Psi_\sigma(\bl z),{\bl w}\big),
\eea
where $\m B_i$ is the Bergman kernel of the domain $\Omega_i,\,\,\bl \Psi\in {\rm Aut}(\Omega_1)$ and $\bl \Psi_\sigma = \bl \Psi^{-1} \circ \sigma^{-1} \circ \bl \Psi$ for $\sigma \in G.$
\end{thm}
\begin{proof}
The Bergman kernel $\m B_1$ transforms under the automorphism $\bl \Psi \in {\rm Aut}(\Omega_1)$ following \cite[p. 419, Proposition 12.1.10]{MR3114789}:
\Bea
\m B_1\big(\bl \Psi(\bl z),\bl \Psi(\bl w)\big) =\frac{1}{J_{\bl \Psi}(\bl z) \ov{J_{\bl \Psi}(\bl w)}}  \m B_1(\bl z,\bl w)   \mbox{~for~} \bl z, \bl w\in \Omega_1.
\Eea
  Therefore, for a fixed $\bl w \in \Omega_1,$ \bea \label{seventh}
{(\m B_1)}_{\bl \Psi(\bl w)}(\bl z) \nonumber&=& \frac{1}{\ov{J_{\bl \Psi}(\bl w)}J_{\bl \Psi}\big(\bl \Psi^{-1}(\bl z)\big)} {(\m B_1)}_{\bl w}\big(\bl \Psi^{-1}(\bl z)\big)\\ 
{(\m B_1)}_{\bl \Psi(\bl w)}(\sigma^{-1} \cdot \bl z) \nonumber&=& \frac{1}{\ov{J_{\bl \Psi}(\bl w)}J_{\bl \Psi}\big(\bl \Psi^{-1}(\sigma^{-1} \cdot \bl z)\big)} {(\m B_1)}_{\bl w}\big(\bl \Psi^{-1}( \sigma^{-1} \cdot \bl z)\big) \\ {(\m B_1)}_{\bl \Psi(\bl w)}\big(\sigma^{-1} \cdot \bl \Psi(\bl z)\big) \nonumber&=& \frac{1}{\ov{J_{\bl \Psi}(\bl w)}J_{\bl \Psi}\big(\bl \Psi^{-1} \circ \sigma^{-1} \circ \bl \Psi( \bl z)\big)}  {(\m B_1)}_{\bl w}\big(\bl \Psi^{-1} \circ \sigma^{-1} \circ \bl \Psi(\bl z)\big) .
\eea

Let  $\bl \Psi_\sigma := \bl \Psi^{-1} \circ \sigma^{-1} \circ \bl \Psi.$ Then $\bl \Psi \circ \bl \Psi_\sigma = \sigma^{-1} \circ \bl \Psi.$ Application of the chain rule on the both sides yields 
$J_{\bl \Psi}\big(\bl \Psi_\sigma(\bl z)\big) J_{\bl \Psi_\sigma}(\bl z) = J_{\sigma^{-1}}\big(\bl \Psi(\bl z)\big) J_{\bl \Psi}(\bl z).$ Hence we have
\bea\label{next} \m B_1\big(\sigma^{-1} \cdot \bl \Psi(\bl z), \bl \Psi(\bl w)\big)  &=& \frac{J_{\bl \Psi_\sigma}(\bl z)}{\ov{J_{\bl \Psi}(\bl w)}J_{\bl \Psi}(\bl z)J_{\sigma^{-1}}\big(\bl \Psi(\bl z)\big)}  \m B_1\big(\bl \Psi_\sigma(\bl z),{\bl w}\big).\eea

Suppose that $\m B_{\bl \theta}$ denotes the Bergman kernel of the domain $\bl \theta(\Omega_1).$ From Equation \eqref{bajee}, we have 
\Bea
\m B_{\bl \theta}\big(\bl \theta\big(\bl \Psi(\bl z)\big), \bl \theta(\bl \Psi(\bl w))\big) = \frac{1}{J_{\bl \theta}\big(\bl \Psi(\bl z)\big)\overline{J_{\bl \theta}\big(\bl \Psi(\bl w)\big)}} \displaystyle\sum_{\sigma \in G} \det(\sigma ) \m B_1\big(\sigma^{-1} \cdot \bl \Psi(\bl z), \bl \Psi(\bl w)\big),
\Eea
for $\bl z, \bl w \in \Omega_1$. Equation \eqref{next} implies that \bea \label{eigth} \nonumber&& \m B_{\bl \theta}\big(\bl \theta(\bl \Psi(\bl z)), \bl \theta(\bl \Psi(\bl w))\big) \\ \nonumber&=& \frac{1}{J_{\bl \theta}(\bl \Psi(\bl z))\overline{J_{\bl \theta}\big(\bl \Psi(\bl w)\big)}}  \displaystyle\sum_{\sigma \in G} \det(\sigma) \frac{J_{\bl \Psi_\sigma}(\bl z)}{\ov{J_{\bl \Psi}(\bl w)}J_{\bl \Psi}(\bl z)J_{\sigma^{-1}}\big(\bl \Psi(\bl z)\big)}  \m B_1\big(\bl \Psi_\sigma(\bl z),{\bl w}\big)\\ \nonumber&=& \frac{1}{J_{\bl \theta \circ \bl \Psi}(\bl z)\ov{J_{\bl \theta \circ \bl \Psi}(\bl w)}} \displaystyle\sum_{\sigma \in G} \det(\sigma)  \frac{J_{\bl \Psi_\sigma}(\bl z)}{J_{\sigma^{-1}}\big(\bl \Psi(\bl z)\big)}\m B_1\big(\bl \Psi_\sigma(\bl z),{\bl w}\big) .\eea
Since $\bl F = \bl \Phi \circ \bl \theta \circ \bl \Psi,$ where $\bl \Phi : \bl \theta(\Omega_1) \to \Omega_2$ is a biholomorphism, we have the following:
\Bea \nonumber&&\m B_2\big(\bl F(\bl z), \bl F(\bl w)\big) \\ &=& \frac{1}{J_{\bl F}(\bl z)\ov{J_{\bl F}(\bl w)}}   \displaystyle\sum_{\sigma \in G} \det(\sigma) \frac{J_{\bl \Psi_\sigma}(\bl z)}{J_{\sigma^{-1}}(\bl \Psi(\bl z))}\m B_1\big(\bl \Psi_\sigma(\bl z),{\bl w}\big)  .\Eea
Note that for every $\bl z \in \Omega_1,$ $J_{\sigma^{-1}}(\bl \Psi(\bl z)) = \det(\sigma)$, so $\frac{\det(\sigma)}{J_{\sigma^{-1}}(\bl \Psi(\bl z))} = 1.$ This completes the proof.
\end{proof}
\begin{rem}\rm
The map $ \bl \iota: \Omega_1 \to \Omega_1$ given by $\bl \iota(\bl z) = \bl z$, is indeed an automorphism of the domain $\Omega_1.$ A trivial observation is that  Equation \eqref{maineq2} coincides with Equation \eqref{bajee} for $\bl \Psi = \bl \iota$.
\end{rem}
\section{Applications}\label{app}
The Bergman kernel on the domain $\Omega \subseteq \mb C^d$ is denoted by $\m B_{\Omega}.$ We fix this notation for the rest of our discussion. Let $\mb B_d$ be the unit ball with respect to the $\ell^2$-norm induced by the standard inner product $\langle \cdot,\cdot\rangle$ on $\mb C^d.$ The Bergman kernel of $\mb B_d$ is given by \cite[p. 172]{MR1375158}
\bea \m B_{\mb B_d}(\bl z,\bl w)&=& \big(1-\langle \bl z,\bl w\rangle\big)^{-(d+1)}.\eea

Let $\mb D^d=\{\bl z \in \mb C^d: |z_1|, \ldots , |z_d| < 1\}$, the unit ball with respect to $\ell^\infty$-norm, be the polydisc in $\mb C^d$. The Bergman kernel of $\mb D^d$ is given by \bea\label{bi} \m B_{\mb D^d}(\bl z,\bl w)&=&\prod_{j=1}^d (1-z_jw_j)^{-2}. \eea
In this section, we obtain formulae for the weighted Bergman kernels of  several domains which are biholomorphically equivalent to some quotient domains of the form $\Omega /G,$ where $\Omega = \mb B_d$ or $\mb D^d$ and $G$ is a finite complex reflection group. This demonstrates an useful application of Theorem \ref{firstmain} and Theorem \ref{main2}. 

Suppose that for a finite complex reflection group $G,$ a basic polynomial map associated to the group $G$ is denoted by $\bl \theta.$ Clearly, $J_{\bl \theta}$ is again a polynomial. If $\mb B_d$ (or $\mb D^d$) is $G$-invariant, then the kernel function $\m B_{\bl \theta(\mb B_d)}(\bl z,\bl w)$ (or $\m B_{\bl \theta(\mb D^d)}(\bl z,\bl w)$) is rational from Corollary \ref{ratcat}. In this section, we consider domains (except Rudin's domains) which can be realized as $\bl \theta(\mb D^d)$ for a basic polynomial map $\bl \theta$ associated to some finite complex reflection group $G.$ This provides classes of domains with rational Bergman kernels. 

\subsection{Rudin's domain}
Rudin  characterizes proper holomorphic mappings from $\mb B_d$ onto a domain $\Omega \subset \mathbb{C}^d, d>1,$ in \cite[p. 704, Theorem 1.6]{MR667790}, see also \cite[p. 506]{MR807287}. The result is stated as:
\begin{thm*} \cite[p. 704, Theorem 1.6]{MR667790}
Suppose that $ \bl F : \mb B_d \to \Omega$ is a proper holomorphic mapping from the open unit ball $\mb B_d$ in $\mb C^d (d>1)$ onto a domain $\Omega$ in $\mb C^d$ with multiplicity $m>1.$ Then there exists a unique finite complex reflection group $G$ of order $m$  such that 
\bea\label{factor}
\bl F = \bl \Phi \circ \bl \theta \circ \bl \Psi ,
\eea
where $\bl \Psi $ is an automorphism of $\mb B_d,$ $\bl \theta$ is a basic polynomial mapping associated to $G$ and $\bl \Phi:\bl\theta(\mb B_d)\to\Omega$ is  biholomorphic.
\end{thm*}
In other words, any proper holomorphic mapping from the open unit ball $\mb B_d$ onto the domain $\Omega$ is factored by automorphisms $\widetilde{G},$ where $\widetilde{G}$ is a conjugate to the finite complex reflection group $G$ by the automorphism $\bl \Psi \in {\rm Aut}(\mb B_d).$ Such a domain $\Omega$ is referred as Rudin's domain in \cite[p. 427]{MR742433}. Here, we include formulae for the Bergman kernels of Rudin's domains to exhibit a direct application of Theorem \ref{main2}. 

The group of unitary operators on $\mb C^d,$ $\m U_d,$ leaves the open unit ball $\mb B_d$ invariant. Since any complex reflection group $G$ acting on $\mb C^d$ is a subgroup of $\m U_d,$  $\mb B_d$ is $G$-invariant. Now we use Theorem \ref{main2} to get the following result.
\begin{thm}
Suppose that $ \bl F : \mb B_d \to \Omega$ is a proper holomorphic mapping. Then the Bergman kernel $\m B_{\Omega}$ is given by the following formula : \bea\label{final}  \m B_{\Omega}\big(\bl F(\bl z), \bl F(\bl w)\big) = \frac{1}{J_{\bl F}(\bl z) \ov{J_{\bl F}(\bl w)}} \displaystyle\sum_{\sigma \in G}  \frac{J_{\bl \Psi_\sigma}(\bl z)}{\big(1 - \langle \bl \Psi_\sigma(\bl z), \bl w \rangle \big)^{d+1}},\qquad \eea where $\Psi\in {\rm Aut}(\mb B_d)$ and $\bl \Psi_\sigma = \bl \Psi^{-1}\sigma^{-1} \bl \Psi$ for $\sigma \in G.$
\end{thm}

\subsection{Symmetrized Polydisc}
The permutation group on $d$ symbols is denoted by $\mathfrak S_d$. The group $\mathfrak S_d$ acts on $\mb C^d$ by permuting its coordinates, that is, 
\Bea
\sigma \cdot (z_1,\ldots,z_d) = (z_{\sigma^{-1}(1)},\ldots,z_{\sigma^{-1}(d)}) \,\, \text{for}\,\,  \sigma \in \mathfrak S_d\,\, \text{and}\,\, (z_1,\ldots,z_d) \in \mb C^d.
\Eea
Clearly, the open unit polydisc $\mb D^d$ is invariant under the action of the group $\mathfrak S_d.$   

Let $s_k$ denote the elementary symmetric polynomials of degree $k$ in $d$ variables, for $k=1,\ldots, d.$ 
The symmetrization map $\bl s :=(s_1,\ldots,s_d) : \mb C^d \to \mb C^d$ is a basic polynomial map associated to the complex reflection group $\mathfrak S_d.$ The domain $\mb G_d:=\bl s(\mb D^d)$ is known as the symmetrized polydisc.  

The weight function $\omega(\bl z) = \prod_{j=1}^d (1-|z_j|^2)^{\l -2}$ is $\mathfrak{S}_d$-invariant, so there exists continuous function $\widetilde{\omega}:\mb G_d \to (0,\infty)$ such that $\omega = \widetilde{\omega} \circ \bl s.$ For $\l >1,$ the reproducing kernel of the weighted Bergman space $\mb A^2_\omega(\mb D^d)$ is $\m B_{\mb D^d}^{(\l)}(\bl z,\bl w)=\prod_{j=1}^d (1-z_jw_j)^{-\l}.$

The symmetric group $\mathfrak{S}_d$ has only two one-dimensional representations in $\widehat{\mathfrak{S}}_d.$ Those are the sign representation and the trivial representation of $\mathfrak{S}_d$.

{\fontfamily{qpl}\selectfont For the sign representation.} The relative invariant subspace subspace $R^{\mathfrak{S}_d}_\sgn(\mb A_\omega^2(\mb D^d))$ is isometrically isomorphic to the Bergman space $\mb A_{\widetilde{\omega}}^2(\mb G_d)$(cf. Theorem \ref{firstmain}). Therefore, the weighted Bergman kernel $\m B_{\widetilde{\omega}}$ of $\mb A_{\widetilde{\omega}}^2(\mb G_d)$  can be derived in terms of the weighted Bergman kernel of $\m B_{\mb D^d}^{(\l)}.$ From Equation \eqref{la},  the weighted Bergman kernel $\m B_{\widetilde{\omega}}$ is given by the  formula:
\Bea
\m B_{\widetilde{\omega}}\big(\bl s(\bl z), \bl s(\bl w)\big) &=& \frac{1}{J_{\bl s}(\bl z) \ov{J_{\bl s}(\bl w)}} \displaystyle\sum_{\sigma \in \mathfrak S_d} \det(\sigma) \m B_{\mb D^d}^{(\l)}(\sigma^{-1} \cdot \bl z, \bl w)    .\Eea
Note that  $\det(\sigma) = {\rm sgn}(\sigma^{-1}),$ for $\sigma \in \mathfrak S_d.$ 
Therefore, 
\Bea \m B_{\widetilde{\omega}}\big(\bl s(\bl z), \bl s(\bl w)\big)&=& \frac{1}{J_{\bl s}(\bl z) \ov{J_{\bl s}(\bl w)}} \displaystyle\sum_{\sigma \in \mathfrak S_d} {\rm sgn}(\sigma^{-1}) \m B_{\mb D^d}^{(\l)}( \bl z, \sigma \cdot \bl w)   \\  &=& \frac{1}{J_{\bl s}(\bl z) \ov{J_{\bl s}(\bl w)}} \displaystyle\sum_{\sigma \in \mathfrak S_d} {\rm sgn}(\sigma^{-1}) \prod_{i=1}^d(1-z_i\bar{w}_{\sigma^{-1}(i)})^{-\l}   \\ &=& \frac{1}{J_{\bl s}(\bl z) \ov{J_{\bl s}(\bl w)}} \det \big(\!\!\big( (1 - z_i \bar{w_j})^{-\l}\big)\!\!\big)_{i,j =1 }^{d} .
\Eea
Note that  $J_{\bl s}(\bl z)= \displaystyle\prod_{i<j} (z_i-z_j)$ \cite[p. 370, Lemma 10]{MR2135687}. \begin{prop}\label{anti}
The weighted Bergman kernel of $\mb A_{\widetilde{\omega}}^2(\mb G_d)$ is given by $$\m B_{\widetilde{\omega}}\big(\bl s(\bl z), \bl s(\bl w)\big)=\frac{\det \big(\!\!\big( (1 - z_i \bar{w_j})^{-\l}\big)\!\!\big)_{i,j =1 }^{d}}{\displaystyle\prod_{i<j} (z_i-z_j) (\overbar{w_i}-\overbar{w_j})},\,\, \bl z, \bl w \in \mb D^d.$$
\end{prop} In particular, for $\l =2$ the weight function $\omega \equiv 1$ (consequently, $\widetilde{\omega} \equiv 1$) and the Bergman kernel of the symmetrized polydisc $\mb G_d$ can be deduced from the above result which is \Bea
\m B_{\mb G_d}(\bl s(\bl z), \bl s(\bl w)) = \frac{\det \big(\!\!\big( (1 - z_i \bar{w_j})^{-2}\big)\!\!\big)_{i,j =1 }^{d}}{\displaystyle\prod_{i<j} (z_i-z_j) (\overbar{w_i}-\overbar{w_j})}. 
\Eea This expression of the Bergman kernel of the symmetrized polydisc was obtained  in \cite[p. 369, Proposition 9]{MR2135687} using Bell's transformation rule on the set of regular values of the symmetrization map $\bl s.$ A different approach is followed to derive this formula in \cite[p. 2366, Theorem 2.3]{MR3043017}. 

{\fontfamily{qpl}\selectfont For the trivial representation.} The trivial representation of $\mathfrak{S}_d$ is given by ${\rm tr} : \mathfrak{S}_d \to \mb C^*$ such that ${\rm tr}(\sigma) = 1$ for all $\sigma \in \mathfrak{S}_d.$ The generating polynomial $\ell_{\rm tr}$ is a constant polynomial, so we can choose $\ell_{\rm tr} \equiv 1.$ Then from Theorem \ref{firstmain}, the relative invariant subspace $R^{\mathfrak{S}_d}_{\rm tr}(\mb A_\omega^2(\mb D^d))$ is isometrically isomorphic to the Bergman space $\mb A_{\widetilde{\omega}_1}^2(\mb G_d),$ where the weight function is given by \bea\label{eqight}\widetilde{\omega}_1(\bl s(\bl z)) = \frac{1}{\prod_{i<j} |z_i-z_j|^2} \widetilde{\omega}(\bl s(\bl z)).\eea An explicit expression for the reproducing kernel $\m B_{\widetilde{\omega}_1}$ of the weighted Bergman space $\mb A_{\widetilde{\omega}_1}^2(\mb G_d)$ is derived using Equation \eqref{formu}:
\Bea  \m B_{\widetilde{\omega}_1}(\bl s(\bl z), \bl s(\bl w)) &=& \displaystyle\sum_{\sigma \in \mathfrak S_d}  \m B^{(\l)}_{\mb D^d}(\sigma^{-1} \cdot \bl z, \bl w) \\ &=& \displaystyle\sum_{\sigma \in \mathfrak S_d} \prod_{i=1}^d(1-z_i\bar{w}_{\sigma^{-1}(i)})^{-\l} \\ &=&  {\rm perm}\big(\!\!\big( (1 - z_i \bar{w_j})^{-\l}\big)\!\!\big)_{i,j =1 }^{d}, \Eea
where ${\rm perm}A$ denotes the permanent of the matrix $A.$
\begin{prop}\label{sym}
Let ${\widetilde{\omega}_1} : \mb G_d \setminus \bl s(N) \to (0,\infty)$ be the continuous function defined as in Equation \eqref{eqight}, where $N = \{\bl z \in \mb D^d : z_i = z_j \text{~for at least two~} i,j, i\neq j\}$. The reproducing kernel $\m B_{\widetilde{\omega}_1}$ of the weighted Bergman space $\mb A_{\widetilde{\omega}_1}^2(\mb G_d)$ is given by  \Bea
\m B_{\widetilde{\omega}_1}(\bl s(\bl z), \bl s(\bl w)) = {\rm perm}\big(\!\!\big( (1 - z_i \bar{w_j})^{-\l}\big)\!\!\big)_{i,j =1 }^{d},\,\, \bl z, \bl w \in \mb D^d, \Eea
where ${\rm perm}A$ denotes the permanent of the matrix $A.$
\end{prop}

\subsection{The quotient of the unit bidisc by the dihedral group} ($\mathbb D^2 / D_{2k}$)
Let $D_{2k} = \langle \delta, \sigma : \delta^k=\sigma^2 ={{\rm id}, \sigma \delta \sigma^{-1} = \delta^{-1}} \rangle$ be the dihedral group of order $2k.$ We define its action on $\C^2$ via the faithful representation $\pi$ defined by
\Bea
     \pi : D_{2k} \to GL(2,\mb C): \delta\mapsto \begin{bmatrix}
    \omega_k & 0\\
    0 & {\omega_k}^{-1}
    \end{bmatrix}, \sigma\mapsto \begin{bmatrix}
    0 & 1\\
    1 & 0
    \end{bmatrix},
    \Eea
    where $\omega_k$ denotes a primitive $k$-th root of unity. We write the matrix representation of the group action with respect to the standard basis of $\C^2.$ The polynomial map $\bl \phi(z_1,z_2) = (z_1^k + z_2^k, z_1z_2)$ for $(z_1,z_2) \in \mb C^2$, is a basic polynomial map associated to the complex reflection group $D_{2k}.$ The open unit bidisc $\mb D^2$ is invariant under this action. The restriction map $\bl \phi: \mb D^2 \to \bl \phi(\mb D^2):=\m D_{2k}$ is a proper holomorphic map of multiplicity $2k$ and is factored by automorphisms $D_{2k}.$ Clearly, $J_{\bl \phi}(z_1,z_2) = k(z_1^k -z_2^k).$
    
    The number of one-dimensional representations of the dihedral group $D_{2k}$ in $\widehat{D}_{2k}$ is $2$ if $k$ is odd and $4$ if $k$ is even. Clearly, for every $k \in \mb N$ the trivial representation of $D_{2k}$ and the sign representation of $D_{2k}$ are in $\widehat{D}_{2k}.$
  
  {\fontfamily{qpl}\selectfont For the sign representation.} Since the relative invariant subspace subspace $R^{D_{2k}}_\sgn(\mb A^2(\mb D^2))$ is isometrically isomorphic to the Bergman space $\mb A^2(\m D_{2k}),$ we derive the Bergman kernel of the domain $\m D_{2k}$ using Corollary \ref{parti}. Note that
 \Bea
 \det(\sigma^i \delta^j) =  \begin{cases}              \phantom{-}1& \text{if } i=0, \\
              -1& \text{if } i=1,
        \end{cases}
        \Eea
        for $j=0,\ldots,k-1.$  Recalling  Equation \eqref{bi}, $\m B_{\mb D^2}(\bl z, \bl w) = \frac{1}{(1- z_1\overbar{w_1})^2(1-  z_2\overbar{w_2})^2},$ we conclude form Equation \eqref{thir} that: \Bea&& \m B_{\m D_{2k}}\big(\bl \phi(z_1,z_2),\bl \phi(w_1,w_2)\big) 
        \\ &=& \frac{1}{k^2(z_1^k -z_2^k)(\overline{w_1}^k -\overline{w_2}^k)} \sum_{i=1}^k \left( \m B_{\mb D^2}( \delta^i \cdot \bl z, \bl w) - \m B_{\mb D^2}(\sigma \delta^i \cdot \bl z, \bl w) \right) \\ &=& \frac{1}{k^2(z_1^k -z_2^k)(\overline{w_1}^k -\overline{w_2}^k)}  \times \\ && \sum_{i=1}^k\left( \frac{1}{(1- \omega_k^i z_1\overbar{w_1})^2(1- \omega_k^{k-i} z_2\overbar{w_2})^2} - \frac{1}{(1- \omega_k^i z_2\overbar{w_1})^2(1- \omega_k^{k-i} z_1\overbar{w_2})^2}\right). \Eea
        After a tedious but straightforward calculation, we state the following proposition.
        \begin{prop}
        The Bergman kernel of $\m D_{2k}$ is given by the following formula:
        \Bea&&\m B_{\m D_{2k}}\big(\bl \phi(z_1,z_2),\bl \phi(w_1,w_2)\big) \\&=& \frac{(z_1-z_2)}{k^2(z_1^k -z_2^k)(\overline{w_1}^k -\overline{w_2}^k)}  \times \\ && \sum_{i=1}^k\frac{\big(2(1+z_1z_2\overline{w_1}\overline{w_2}) -(z_1+z_2)(\omega_k^i\overline{w_1} + \omega_k^{k-i}\overline{w_2})\big)(\omega_k^i\overline{w_1} - \omega_k^{k-i}\overline{w_2})}{\big(X_1X_2 - (\omega_k^i\overline{w_1} + \omega_k^{k-i}\overline{w_2})(z_2X_1+z_1X_2) +z_1z_2(\omega_k^i\overline{w_1} + \omega_k^{k-i}\overline{w_2})^2\big)^2},\Eea
        where $\omega_k$ is a primitive $k$-th root of unity and 
        $X_\ell=1+z_\ell^2\overline{w_1}\overline{w_2},$ for $\ell=1,2.$
        \end{prop}
 {\fontfamily{qpl}\selectfont For the trivial representation.} The trivial representation of $D_{2k}$ is given by ${\rm tr} : D_{2k} \to \mb C^*$ such that ${\rm tr}(\sigma) = 1$ for all $\sigma \in D_{2k}.$ The generating polynomial $\ell_{\rm tr}$ can be taken as $\ell_{\rm tr} \equiv 1.$ From Theorem \ref{firstmain}, we get that the relative invariant subspace $R^{D_{2k}}_{\rm tr}(\mb A^2(\mb D^2))$ is isometrically isomorphic to the Bergman space $\mb A_\omega^2(\m D_{2k}),$ where the weight function is given by \bea\label{eqight1}\omega(\bl \phi(\bl z)) = \frac{1}{k(z_1^k -z_2^k)}.\eea Hence we have the following from Equation \eqref{formu}:
 \begin{prop}
Let $\omega : \m D_{2k} \setminus \bl \phi(N) \to (0,\infty)$ be the continuous function defined as in Equation \eqref{eqight1}, where $N = \{\bl z \in \mb D^2 : z_1 = \omega_k z_2,~ \omega_k \text{~is a $k$-th root of unity} \}$. The reproducing kernel $\m B_{\m D_{2k}}^\omega$ of the weighted Bergman space $\mb A_\omega^2(\m D_{2k})$ is given by  \Bea
&&\m B_{\m D_{2k}}^\omega\big(\bl \phi(z_1,z_2),\bl \phi(w_1,w_2)\big)\\ &=& \sum_{i=1}^k\left( \frac{1}{(1- \omega_k^i z_1\overbar{w_1})^2(1- \omega_k^{k-i} z_2\overbar{w_2})^2} + \frac{1}{(1- \omega_k^i z_2\overbar{w_1})^2(1- \omega_k^{k-i} z_1\overbar{w_2})^2}\right),\,\, \bl z, \bl w \in \mb D^2. \Eea
\end{prop}

{\fontfamily{qpl}\selectfont For two additional one-dimensional representations while $k=2k'$.} While $k$ is even, $D_{2k}$ has two more one-dimensional representations. We refer those by $\varrho_1$ and $\varrho_2,$ where \Bea \varrho_1(\delta) = -1 &\text{~and~}& \varrho_1(\tau) =1 \text{~for~} \tau \in \inner{\delta^2}{\sigma}, \\ \varrho_2(\delta) = -1 &\text{~and~}& \varrho_2(\tau) =1 \text{~for~} \tau \in \inner{\delta^2}{\delta\sigma}.\Eea The generating polynomials associated to $\varrho_1$ and $\varrho_2$ are given by $\ell_{\varrho_1}(\bl z) = z_1^{k'} + z_2^{k'}$ and $\ell_{\varrho_2}(\bl z) = z_1^{k'} - z_2^{k'},$ respectively. Let the weight functions $\widehat{\omega}_1 : \m D_{2k} \setminus \bl \phi(N_1) \to (0,\infty)$ and $\widehat{\omega}_2: \m D_{2k} \setminus \bl \phi(N_2) \to (0,\infty)$ be defined as the following:
\Bea \widehat{\omega}_1(\bl \phi(\bl z)) = \frac{1}{k(z_1^{k'} -z_2^{k'})} \text{~~and~~} \widehat{\omega}_2(\bl \phi(\bl z)) = \frac{1}{k(z_1^{k'} +z_2^{k'})},\Eea where $N_1 =\{\bl z \in \mb D^2 : z_1 = \omega_{k'} z_2,~ \omega_{k'} \text{~is a $k'$-th root of unity} \}$ and $N_2 = \{\bl z \in \mb D^2 : z^{k'}_1 + z^{k'}_2 =0 \}.$ Now a similar approach as above (using Theorem \ref{firstmain}) will lead to explicit expressions for the reproducing kernels of the weighted Bergman spaces $\mb A_{\widehat{\omega}_1}^2(\m D_{2k})$ and $\mb A_{\widehat{\omega}_2}^2(\m D_{2k})$.


\subsection{Monomial Polyhedron}
 
  For $d \geq 2,$ a $d$-tuple $\bl \alpha = (\alpha_1,\ldots,\alpha_d) \in \mb Q^d$ of rational numbers and a $d$-tuple of complex numbers $\bl z = (z_1,\ldots,z_d) \in \mb C^d$, we denote $\bl z^{\bl \alpha} := \displaystyle\prod_{k=1}^d z_k^{\alpha_k}.$ 
 Consider a matrix $B \in M_d(\mb Q).$  We enumerate the row vectors of $B$ by $\mathcal F = \{\bl b^1,\ldots,\bl b^d\},$ where $\bl b^k = (b_1^k,\ldots,b_d^k).$ The monomial polyhedron associated to $B$ is defined by $$\mathscr U = \{ \bl z \in \mb C^d : |{\bl z}^{\bl b^k}| <1 \text{~for all ~} 1\leq k \leq d\},$$ unless for some $1 \leq k,j \leq d$, the quantity $ z_j^{b_j^k}$ is not defined due to the division of zero, see \cite[Equation 1.1]{bender2020lpregularity}. Without loss of generality, we  assume that $B \in M_d(\mb Z),~ \det(B)>0$ and $B^{-1} \succeq 0$ \cite[Equation 3.3]{bender2020lpregularity}.  
 
 Set  $A = \adj B.$ The rows of $A$ are enumerated by $\{\bl a^1,\ldots,\bl a^d\}$. 
We borrow
the proper holomorphic map $\bl \Phi_A : \mb D^d_{L(B)} \to \mathscr U$ defined by $$ \bl \Phi_A (\bl z) = (\bl z^{\bl a^1},\ldots,\bl z^{\bl a^d}),\,\, \text{~for~} \bl z \in \mb D^d_{L(B)}$$ from \cite[Theorem 3.12]{bender2020lpregularity}, where $\mb D^d_{L(B)}$ is the product of some copies of the unit disc with some copies of the
punctured unit disc. The proper map $\bl \Phi_A$ is of quotient type with complex reflection group $G$ and $G$ is isomorphic to the direct product of cyclic groups $\prod_{i=1}^d \mb Z / {\delta i} \mb Z,$ where each $\delta_i \in \mb Z$ is coming from the Smith Normal form of the matrix $A,$ that says, $A = PDQ,$ where $P,Q \in GL_d(\mb Z)$ and $D ={\rm diag}(\delta_1,\ldots,\delta_d) \in M_d(\mb Z).$ Let $\wp : \prod_{i=1}^d \mb Z / {\delta i} \mb Z \to G$ be a group isomorphism. Set $S_G = \{\bl n = (n_1,\ldots,n_d) \in \mb N^d : 1 \leq n_i \leq |\delta_i| \}.$ We write 
\Bea
G= \{\sigma_{\bl n} :   \wp\big(\prod_{i=1}^d \omega_{|\delta_i|}^{n_i}\big) = \sigma_{\bl n} \text{~for~} \bl n = (n_1,\ldots,n_d) \in S_G\},
\Eea
where $\omega_{|\delta_i|}$ is a primitive $|\delta_i|$-th root of unity for $i=1,\ldots,d$. 
\begin{prop}
The Bergman kernel $\m B_{\mathscr U}$ of the monomial polyhedron $\mathscr U$ is given by: \bea\label{bergU}
\m B_{\mathscr U}\big(\bl \Phi_A(\bl z), \bl \Phi_A(\bl w)\big) = \frac{1}{(\det A)^2} \cdot \frac{\prod_{i=1}^d z_i\overline{w_i}}{\prod_{i=1}^d \bl z^{\bl a^i}\overline{\bl w}^{\bl a^i}} 
\displaystyle\sum_{\bl n \in S_G} \prod_{i=1}^d \omega_{|\delta_i|}^{n_i} \m B_{\mb D^d}(\sigma_{\bl n}^{-1} \cdot \bl z, \bl w), 
\eea
where $\m B_{\mb D^d}$ is the Bergman kernel of $\mb D^d.$
\end{prop}
\begin{proof}
We use Corollary \ref{main} to get the following: 
\Bea
\m B_{\mathscr U}\big(\bl \Phi_A(\bl z), \bl \Phi_A(\bl w)\big) = \frac{1}{J_{\bl \Phi_A}(\bl z)\overline{J_{\bl \Phi_A}(\bl w)}} \displaystyle\sum_{\sigma_{\bl n} \in G} \det(\sigma_{\bl n}) \widehat{\m B_{\mb D^d}}(\sigma_{\bl n}^{-1} \cdot \bl z, \bl w) ,
\Eea
 where the Bergman kernel of $\mb D^d_{L(B)}$ is denoted by $\widehat{\m B_{\mb D^d}}$. Note that $\widehat{\m B_{\mb D^d}}(\bl z, \bl w) = \m B_{\mb D^d}(\bl z, \bl w),$ whenever $\bl z, \bl w \in \mb D^d_{L(B)}.$ 
 Since the character remains unchanged under group isomorphism, the representation $\mu$ of $G$, as described in Equation \eqref{mu}, gives $\chi_\mu(\sigma^{-1}_{\bl n} ) = \det(\sigma_{\bl n}) = \prod_{i=1}^d \omega_{|\delta_i|}^{n_i}$ for every $\bl n \in S_G.$
 Moreover, from \cite[Lemma 3.8]{bender2020lpregularity} \bea \label{det} J_{\bl \Phi_A}(\bl z) = \det A \cdot \frac{\prod_{i=1}^d \bl z^{\bl a^i}}{\prod_{i=1}^d z_i}. \eea Thus the result follows.
\end{proof}
In \cite[Proposition 3.22]{bender2020lpregularity}, Chakrabarti et al. proved rationality of the Bergman kernel of the monomial polyhedron using Bell's transformation rule for the Bergman kernels under a proper holomorphic mapping. This can also be seen from Equation \eqref{bergU} and Chevalley-Shephard-Todd theorem, since the denominator of the sum is a $G$-invariant polynomial both in $\bl \Phi_A(\bl z)$ and $\overbar{\bl \Phi_A(\bl w)}$ and so is the numerator of the sum after dividing by $J_{\bl \Phi_A}(\bl z)\overline{J_{\bl \Phi_A}(\bl w)}.$ 


\subsubsection{Fat Hartogs Triangle.} Let $\gamma$ be a positive integer. For the matrix $B = \begin{pmatrix}
\gamma & -1 \\
0 & \phantom{-}1 
\end{pmatrix}$, the domain $\mathscr U$ can be written as $\{(z_1,z_2) \in \mb C^2 : |z_1|^\gamma <|z_2| <1 \}.$ This turns out to be a subclass of `fat Hartogs triangle' \cite[p. 4533]{MR4244876}. We denote such domains by $\Omega_\gamma.$ The map $\Phi_A : \mb D^2_{L(B)} \to \Omega_\gamma$, given by $(z_1,z_2) \mapsto (z_1z_2, z_2^\gamma),$ is a proper holomorphic map which is factored by a group isomorphic to $\mb Z_\gamma.$ From Corollary \ref{main}, we write the Bergman kernel for the domain $\Omega_\gamma$ as following:

\Bea \m B_{\Omega_\gamma}\big(\Phi_A(\bl z),\Phi_A(\bl w)\big)&=& \frac{1}{\gamma^2(z_2\overbar{w_2})^{\gamma}}  \displaystyle\sum_{\sigma \in \mb Z_\gamma} \det(\sigma) \m B_{\mb D^2}(\sigma^{-1} \cdot \bl z, \bl w)\\&=& \frac{1}{\gamma^2(z_2\overbar{w_2})^{\gamma}}  \displaystyle\sum_{k=1}^\gamma \omega_\gamma^k ~ \m B_{\mb D^2}\big((\omega_\gamma^kz_1,\omega_\gamma^{-k}z_2), (w_1,w_2)\big)\\
&=& \frac{1}{\gamma^2(z_2\overbar{w_2})^{\gamma}}  \displaystyle\sum_{k=1}^\gamma  \frac{\omega_\gamma^k}{(1 - \omega_\gamma^kz_1 \ov{w_1} )^2 (1 - \omega_\gamma^{-k}z_2 \ov{w_2})^2},
\Eea
where $\omega_\gamma$ is a primitive $\gamma$-th root of unity. 

\subsection{Complex Ellipsoid}
A subclass of complex ellipsoids is given by \Bea
\Omega_{p,q} := \{ (z_1,z_2) \in \mb C^2 : |z_1|^{2p'} + |z_2|^{2q'} < 1\}\Eea for $p=\frac{1}{p'}, q= \frac{1}{q'} \in \mb N$. 
The holomorphic map $ \Phi : \mb B_2 \to \Omega_{p,q}$ defined by $(z_1,z_2) \mapsto (z_1^p,z_2^q)$ is proper. The map $\Phi$ is factored by automorphisms $G \subseteq {\rm Aut}(\mb B_2),$ where $G$ is isomorphic to the complex reflection group $\mb Z_p \times \mb Z_q.$
We use Corollary \ref{main} to determine a formula for the Bergman kernel of $\Omega_{p,q},$ denoted by $\m B_{\Omega_{p,q}}$. Therefore, the kernel $\m B_{\Omega_{p,q}}$ is given by \Bea &&\m B_{\Omega_{p,q}}\big(\Phi(z_1,z_2),\Phi(w_1,w_2)\big) \\ &=& \frac{1}{(pq)^2 (z_1\overbar{w_1})^{p-1}(z_2\overbar{w_2})^{q-1}} \sum_{\alpha_1 = 1}^p\sum_{\alpha_2 = 1}^q \omega_p^{\alpha_1} \omega_q^{\alpha_2} \m B_{\mb B_2}\big((\omega_p^{\alpha_1}z_1,\omega_q^{\alpha_2}z_2),(w_1,w_2)\big)\\ &=& \frac{1}{(pq)^2 (z_1\overbar{w_1})^{p-1}(z_2\overbar{w_2})^{q-1}} \sum_{\alpha_1 = 1}^p\sum_{\alpha_2 = 1}^q \frac{\omega_p^{\alpha_1} \omega_q^{\alpha_2}}{\big(1 - (\omega_p^{\alpha_1}z_1\ov{w_1} + \omega_q^{\alpha_2}z_2 \ov{w_2})\big)^3},\Eea
where $(z_1,z_2), (w_1,w_2) \in \mb B_2,$ and $\omega_k$ denotes a primitive $k$-th root of unity.

\medskip \textit{Acknowledgment}.
The author would like to express her sincere gratitude to  Shibananda Biswas and Subrata Shyam Roy for several comments and suggestions in the preparation of this article.
 The author acknowledges the financial support of CV Raman postdoctoral fellowship from Indian Institute of Science and postdocotoral fellowship from Silesian University in Opava under GA CR grant no. 21-27941S. Also the research is part of the project No. 2022/45/P/ST1/01028 co-funded by the National Science Centre and the European Union Framework Programme for Research and Innovation Horizon 2020 under the Marie Sklodowska-Curie grant agreement No.945339. 

\subsection*{Conflict of interest} The author declares that there is no conflict of interest. 

\subsection*{Data availability statement}Data sharing is not applicable to this article as no new data were created or analyzed in this study.


\begin{thebibliography}{10}

\bibitem{MR2365665}
{\sc A.~A. Abouhajar, M.~C. White, and N.~J. Young}, {\em A {S}chwarz lemma for
  a domain related to {$\mu$}-synthesis}, J. Geom. Anal., 17 (2007),
  pp.~717--750.

\bibitem{MR3771126}
{\sc J.~Agler, Z.~Lykova, and N.~J. Young}, {\em Algebraic and geometric
  aspects of rational {$\Gamma$}-inner functions}, Adv. Math., 328 (2018),
  pp.~133--159.

\bibitem{MR3912883}
\leavevmode\vrule height 2pt depth -1.6pt width 23pt, {\em A geometric
  characterization of the symmetrized bidisc}, J. Math. Anal. Appl., 473
  (2019), pp.~1377--1413.

\bibitem{MR4293930}
\leavevmode\vrule height 2pt depth -1.6pt width 23pt, {\em Intrinsic
  directions, orthogonality, and distinguished geodesics in the symmetrized
  bidisc}, J. Geom. Anal., 31 (2021), pp.~8202--8237.

\bibitem{MR1674635}
{\sc J.~Agler and N.~J. Young}, {\em A commutant lifting theorem for a domain
  in {$\bold C^2$} and spectral interpolation}, J. Funct. Anal., 161 (1999),
  pp.~452--477.

\bibitem{MR1744711}
\leavevmode\vrule height 2pt depth -1.6pt width 23pt, {\em Operators having the
  symmetrized bidisc as a spectral set}, Proc. Edinburgh Math. Soc. (2), 43
  (2000), pp.~195--210.

\bibitem{MR2077158}
\leavevmode\vrule height 2pt depth -1.6pt width 23pt, {\em The hyperbolic
  geometry of the symmetrized bidisc}, J. Geom. Anal., 14 (2004), pp.~375--403.

\bibitem{MR1375158}
{\sc B.~Bagchi and G.~Misra}, {\em Homogeneous tuples of multiplication
  operators on twisted {B}ergman spaces}, J. Funct. Anal., 136 (1996),
  pp.~171--213.

\bibitem{MR807287}
{\sc E.~Bedford and S.~Bell}, {\em Boundary behavior of proper holomorphic
  correspondences}, Math. Ann., 272 (1985), p.~505–518.

\bibitem{MR807258}
{\sc E.~Bedford and J.~Dadok}, {\em Proper holomorphic mappings and real
  reflection groups}, J. Reine Angew. Math., 361 (1985), p.~162–173.

\bibitem{MR610182}
{\sc S.~R. Bell}, {\em Proper holomorphic mappings and the {B}ergman
  projection}, Duke Math. J., 48 (1981), pp.~167--175.

\bibitem{MR645338}
\leavevmode\vrule height 2pt depth -1.6pt width 23pt, {\em The {B}ergman kernel
  function and proper holomorphic mappings}, Trans. Amer. Math. Soc., 270
  (1982), p.~685–691.

\bibitem{MR742433}
\leavevmode\vrule height 2pt depth -1.6pt width 23pt, {\em Proper holomorphic
  mappings that must be rational}, Trans. Amer. Math. Soc., 284 (1984),
  p.~425–429.

\bibitem{bender2020lpregularity}
{\sc C.~Bender, D.~Chakrabarti, L.~Edholm, and M.~Mainkar}, {\em
  L$^p$-regularity of the {B}ergman projection on quotient domains}, Canadian
  Journal of Mathematics,  (2021), p.~1–37.

\bibitem{BDGS}
{\sc S.~Biswas, S.~Datta, G.~Ghosh, and S.~Shyam~Roy}, {\em Reducing submodules
  of {H}ilbert modules and {C}hevalley-{S}hephard-{T}odd theorem}, Adv. Math.,
  403 (2022).

\bibitem{MR3188714}
{\sc S.~Biswas and S.~Shyam~Roy}, {\em Functional models of
  {$\Gamma_n$}-contractions and characterization of {$\Gamma_n$}-isometries},
  J. Funct. Anal., 266 (2014), p.~6224–6255.

\bibitem{MR1775958}
{\sc H.~P. Boas}, {\em Lu {Q}i-{K}eng's problem}, J. Korean Math. Soc., 37
  (2000), pp.~253--267.
\newblock Several complex variables (Seoul, 1998).

\bibitem{MR1469401}
{\sc H.~P. Boas, S.~Fu, and E.~J. Straube}, {\em The {B}ergman kernel function:
  explicit formulas and zeroes}, Proc. Amer. Math. Soc., 127 (1999),
  pp.~805--811.

\bibitem{MR1890629}
{\sc N.~Bourbaki}, {\em Lie groups and {L}ie algebras. {C}hapters 4–6},
  Elements of Mathematics (Berlin), Springer-Verlag, Berlin, 2002.
\newblock Translated from the 1968 French original by Andrew Pressley.

\bibitem{MR4088498}
{\sc L.~Chen, S.~G. Krantz, and Y.~Yuan}, {\em {$L^p$} regularity of the
  {B}ergman projection on domains covered by the polydisc}, J. Funct. Anal.,
  279 (2020), pp.~108522, 20.

\bibitem{MR1988884}
{\sc X.~Chen and K.~Guo}, {\em Analytic {H}ilbert modules}, vol.~433 of Chapman
  \& Hall/CRC Research Notes in Mathematics, Chapman \& Hall/CRC, Boca Raton,
  FL, 2003.

\bibitem{MR2142182}
{\sc C.~Costara}, {\em On the spectral {N}evanlinna-{P}ick problem}, Studia
  Math., 170 (2005), p.~23–55.

\bibitem{MR4523519}
{\sc G.~Dall'Ara and A.~Monguzzi}, {\em Nonabelian ramified coverings and
  {$L^p$}-boundedness of {B}ergman projections in {$\Bbb C^2$}}, J. Geom.
  Anal., 33 (2023), pp.~Paper No. 52, 28.

\bibitem{MR1131852}
{\sc G.~Dini and A.~Selvaggi~Primicerio}, {\em Proper holomorphic mappings
  between generalized pseudoellipsoids}, Ann. Mat. Pura Appl. (4), 158 (1991),
  p.~219–229.

\bibitem{MR1028546}
{\sc R.~G. Douglas and V.~I. Paulsen}, {\em Hilbert modules over function
  algebras}, vol.~217 of Pitman Research Notes in Mathematics Series, Longman
  Scientific \& Technical, Harlow; copublished in the United States with John
  Wiley \& Sons, Inc., New York, 1989.

\bibitem{MR3107680}
{\sc A.~Edigarian, L.~Kosi\'{n}ski, and W.~Zwonek}, {\em The {L}empert theorem
  and the tetrablock}, J. Geom. Anal., 23 (2013), pp.~1818--1831.

\bibitem{MR2135687}
{\sc A.~Edigarian and W.~Zwonek}, {\em Geometry of the symmetrized polydisc},
  Arch. Math. (Basel), 84 (2005), p.~364–374.

\bibitem{MR338454}
{\sc C.~Fefferman}, {\em On the {B}ergman kernel and biholomorphic mappings of
  pseudoconvex domains}, Bull. Amer. Math. Soc., 80 (1974), pp.~667--669.

\bibitem{gho-gho}
{\sc A.~Ghosh and G.~Ghosh}, {\em {$L^p$} regularity of {S}zeg\"{o} projections
  on quotient domains}, New York J. Math., 29 (2023), pp.~911--930.

\bibitem{ghosh2020multiplication}
{\sc G.~Ghosh}, {\em Multiplication operators on the {B}ergman space by proper
  holomorphic mappings}, J. Math. Anal. Appl., 510 (2022), pp.~Paper No.
  126026, 12.

\bibitem{MR1288523}
{\sc P.~Griffiths and J.~Harris}, {\em Principles of algebraic geometry}, Wiley
  Classics Library, John Wiley \& Sons, Inc., New York, 1994.
\newblock Reprint of the 1978 original.

\bibitem{MR2484092}
{\sc J.~Hilgert}, {\em Reproducing kernels in representation theory}, in
  Symmetries in complex analysis, vol.~468 of Contemp. Math., Amer. Math. Soc.,
  Providence, RI, 2008, pp.~1--98.

\bibitem{MR3114789}
{\sc M.~Jarnicki and P.~Pflug}, {\em Invariant distances and metrics in complex
  analysis}, vol.~9 of De Gruyter Expositions in Mathematics, Walter de Gruyter
  GmbH \& Co. KG, Berlin, extended~ed., 2013.

\bibitem{MR2736338}
{\sc L.~Kosi\'{n}ski}, {\em Geometry of quasi-circular domains and applications
  to tetrablock}, Proc. Amer. Math. Soc., 139 (2011), pp.~559--569.

\bibitem{MR3511461}
{\sc L.~Kosi\'{n}ski and W.~Zwonek}, {\em Nevanlinna-{P}ick problem and
  uniqueness of left inverses in convex domains, symmetrized bidisc and
  tetrablock}, J. Geom. Anal., 26 (2016), pp.~1863--1890.

\bibitem{MR2542964}
{\sc G.~I. Lehrer and D.~E. Taylor}, {\em Unitary reflection groups}, vol.~20
  of Australian Mathematical Society Lecture Series, Cambridge University
  Press, Cambridge, 2009.

\bibitem{MR3043017}
{\sc G.~Misra, S.~Shyam~Roy, and G.~Zhang}, {\em Reproducing kernel for a class
  of weighted {B}ergman spaces on the symmetrized polydisc}, Proc. Amer. Math.
  Soc., 141 (2013), p.~2361–2370.

\bibitem{MR4244876}
{\sc A.~Nagel and M.~Pramanik}, {\em Bergman spaces under maps of monomial
  type}, J. Geom. Anal., 31 (2021), pp.~4531--4560.

\bibitem{MR667790}
{\sc W.~Rudin}, {\em Proper holomorphic maps and finite reflection groups},
  Indiana Univ. Math. J., 31 (1982), p.~701–720.

\bibitem{MR2446682}
\leavevmode\vrule height 2pt depth -1.6pt width 23pt, {\em Function theory in
  the unit ball of {$\Bbb C^n$}}, Classics in Mathematics, Springer-Verlag,
  Berlin, 2008.
\newblock Reprint of the 1980 edition.

\bibitem{Shephard-Todd}
{\sc G.~C. Shephard and J.~A. Todd}, {\em Finite unitary reflection groups},
  Canad. J. Math., 6 (1954), pp.~274--304.

\bibitem{MR460484}
{\sc R.~P. Stanley}, {\em Relative invariants of finite groups generated by
  pseudoreflections}, J. Algebra, 49 (1977), p.~134–148.

\bibitem{MR0473215}
{\sc E.~M. Stein}, {\em Boundary behavior of holomorphic functions of several
  complex variables}, Mathematical Notes, No. 11, Princeton University Press,
  Princeton, N.J.; University of Tokyo Press, Tokyo, 1972.

\bibitem{MR117285}
{\sc R.~Steinberg}, {\em Invariants of finite reflection groups}, Canadian J.
  Math., 12 (1960), p.~616–618.

\bibitem{MR3133729}
{\sc M.~Trybula}, {\em Proper holomorphic mappings, {B}ell's formula, and the
  {L}u {Q}i-{K}eng problem on the tetrablock}, Arch. Math. (Basel), 101 (2013),
  p.~549–558.

\bibitem{MR2418303}
{\sc N.~J. Young}, {\em The automorphism group of the tetrablock}, J. Lond.
  Math. Soc. (2), 77 (2008), pp.~757--770.

\end{thebibliography}
\end{document}